\documentclass[a4paper,10pt,oneside]{article}
\usepackage[latin1]{inputenc}
\usepackage[english]{babel}
\usepackage{amsthm}
\usepackage{amssymb}
\usepackage{amsmath}
\usepackage{hyperref}
\usepackage{graphicx}
\usepackage[all]{xy}
\DeclareMathAlphabet{\mathpzc}{OT1}{pzc}{m}{it}
\newtheoremstyle{note}{11pt}{11pt}{}{}{\bfseries}{.}{.5em}{}
\newtheorem{theo}[equation]{Theorem}
\newtheorem{prop}[equation]{Proposition}

\newtheorem{defin}[equation]{Definition}
\newtheorem{exam}{Example}
\newtheorem{conj}[equation]{Conjecture}
\newtheorem{rem}[equation]{Remark}
\newtheorem{coro}[equation]{Corollary}
\numberwithin{equation}{section}
\newtheorem{lemma}[equation]{Lemma}

\newcommand{\Qb}{\overline{\mathbb{Q}}}
\newcommand{\Q}{\mathbb{Q}}

\newcommand{\T}{\mathbb{T}}
\newcommand{\Aa}{\mathbb{A}}

\newcommand{\Z}{\mathbb{Z}}
\newcommand{\R}{\mathbb{R}}
\newcommand{\C}{\mathbb{C}}

\newcommand{\calE}{\mathcal{E}}

\newcommand{\K}{\mathcal{K}}
\newcommand{\N}{\mathcal{N}}
\newcommand{\Tt}{\mathcal{T}}

\newcommand{\A}{\mathcal{A}}

\newcommand{\M}{\mathcal{M}}

\newcommand{\Ll}{\mathcal{L}}

\newcommand{\Cc}{\mathcal{C}}

\newcommand{\oo}{\mathcal{O}}
\newcommand{\h}{\mathcal{H}}

\newcommand{\W}{\mathcal{W}}
\newcommand{\U}{\mathcal{U}}

\newcommand{\G}{\Gamma}
\newcommand{\LL}{\Lambda}
\newcommand{\eps}{\varepsilon}

\newcommand{\gl}{\mbox{GL}_2}
\newcommand{\Sl}{\mbox{SL}_2}

\newcommand{\beps}{\mathbf{\varepsilon}}

\newcommand{\Dfrak}{\mathfrak{D}}
\newcommand{\pfrak}{\mathfrak{p}}

\newcommand{\lgr}{\left\{}
\newcommand{\rgr}{\right\}}
\newcommand{\rra}{\right\rangle}
\newcommand{\lla}{\left\langle}

\newcommand{\boldE}{\mathbf{E}}
\newcommand{\boldp}{\mathbf{p}}
\newcommand{\boldD}{\mathbf{D}}

\title{A formula for the derivative of the $p$-adic $L$-function of the symmetric square of a finite slope modular form}
\author{Giovanni Rosso
\footnote{PhD Fellowship of Fund for Scientific Research--Flanders, partially supported by a JUMO grant from KU Leuven (Jumo/12/032), a ANR grant  (ANR-10-BLANC 0114 ArShiFo) and a NSF grant (FRG  DMS  0854964). 
}}

\newcommand{\Addresses}{{
  \bigskip
  \footnotesize

  G.~Rosso,\\
   \textsc{Department of Mathematics, KU Leuven\\
Celestijneenlan 200B,\\
3001 Heverlee, Belgium\\}
  {giovanni.rosso@wis.kuleuven.be}
}}

\begin{document}

\maketitle
Let $f$ be a modular form of weight $k$ and Nebentypus $\psi$. By generalizing a construction of \cite{DD}, we construct a $p$-adic $L$-function interpolating the special values of the $L$-function $L(s,\mathrm{Sym}^2(f)\otimes \xi)$, where $\xi$ is a Dirichlet character.\\
When $s=k-1$ and $\xi=\psi^{-1}$, this $p$-adic $L$-function vanishes due to the presence of a so-called trivial zero. We give a formula for the derivative at $s=k-1$ of this $p$-adic $L$-function when the form $f$ is Steinberg at $p$.\\ 
If the weight of $f$ is even, the conductor is even and squarefree, and the Nebentypus is trivial this formula implies a conjecture of Benois.
\tableofcontents
\section{Introduction}
The aim of this paper is to prove a conjecture of Benois on trivial zeros in the particular case of the symmetric square representation of a modular form whose associated automorphic representation at $p$ is Steinberg.\\
We begin by recalling the statement of Benois' conjecture.  Let $G_{\Q}$ be the absolute Galois group of $\Q$.  We fix an odd prime number $p$ and two embeddings 
\begin{align*}
\Qb \hookrightarrow \C_p, \; \; \Qb \hookrightarrow \C,
\end{align*}
and we let $G_{\Q_p}$ be the absolute Galois group of $\Q_p$.
Let 
\begin{align*}
V : G_{\Q} \rightarrow \mathrm{GL}_n(\Qb_p)
\end{align*}
be a continuous, irreducible, $p$-adic Galois representation of $G_{\Q}$. We suppose that $V$ is the $p$-adic realization of a pure motive $M_{/\Q}$ of weight $0$. We can  associate a complex $L$-function $L(s,M)$. Let $M^*=\mathrm{Hom}(M, -)$ be the dual motive of $M$. Conjecturally, if $M$ is not trivial, $L(s,M)$ is a holomorphic function on the whole complex plane satisfying a holomorphic functional equation
\begin{align*}
L(s,M)\G(s,M)= \beps(s,M)L(1-s,M^*)\G(1-s,M^*),
\end{align*}
where $\G(s,M)$ denotes a product of Gamma functions and $\eps(s,M)=\zeta N^s$, for $N$ a positive integer and $\zeta$ a root of unit. We say that $M$ is {\it critical} at $s=0$ if neither $\G(s,M)$ nor $\G(1-s,M^*)$ have a pole at $s=0$.
In this case the complex value $L(s,M)$ is not forced to be $0$ by the functional equation; we shall suppose, moreover, that $L(s,M)$  is not zero. Similarly, we can say that $M$ is critical at an integer $n$ if $s=0$ is critical for $M(n)$.\\
Deligne \cite{Del} has defined a non-zero complex number $\Omega(M)$ (defined only modulo multiplication by a non zero algebraic number) depending only on the Betti and de Rham realizations of $M$, such that conjecturally
\begin{align*}
\frac{L(0,M)}{\Omega(M)} \in \Qb.
\end{align*}
We now suppose  that all the above conjectures are true for $M$ and all its twists $M \otimes \eps$, where $\eps$ ranges among the finite-order characters of $1+ p\Z_p$. We suppose moreover that $V$ is a semi-stable representation of $G_{\Q_p}$. Let $\boldD_{\mathrm{st}}(V)$ be the semistable module associated to $V$; it is a filtered $(\phi,N)$-module, i.e. it is endowed with a filtration and with the action of two operators: a Frobenius $\phi$ and a monodromy operator $N$. We say that a filtered $(\phi,N)$-sub-module $D$ of $\boldD_{\mathrm{st}}(V)$ is regular if 
\begin{align*}
\boldD_{\mathrm{st}}(V) = \mathrm{Fil}^0(\boldD_{\mathrm{st}}(V)) \bigoplus D.
\end{align*}
To these data Perrin-Riou \cite{PR} associates a $p$-adic $L$-function $L_p(s,V,D)$ which is supposed to {\it interpolate} the special values $\frac{L(M\otimes \eps,0)}{\Omega(M)}$, for $\eps$ as above. In particular, it should satisfy
\begin{align*}
L_p(0,V,D)= E_p(V,D)\frac{L(0,M)}{\Omega(M)},
\end{align*}
where $E_p(V,D)$ denotes a finite product of Euler-type factors, corresponding to a subset of the eigenvalues of $\phi$ on $D$ and on the dual regular submodule $D^*$ of $\boldD_{\mathrm{st}}(V^*)$ (see \cite[\S 0.1]{BenEZ}).\\ 
It may happen that some of these Euler factors vanish. In this case, we say that we are in the presence of a {\it trivial zero}. When trivial zeros appear, we would like to be able to retrieve information about the special value $\frac{L(0,M)}{\Omega(M)}$ from the $p$-adic derivative of $L_p(s,V,D)$.\\
Under certain suitable hypotheses (denoted by $\mathbf{C1}-\mathbf{C4}$ in \cite[\S 0.2]{BenLinv2})  on the representation $V$, Benois states the following conjecture:
\begin{conj}\label{MainCoOC}[Trivial zeros conjecture]
Let $e$ be the number of Euler-type factors of $E_p(V,D)$ which vanish. Then 
\begin{align*}
\lim_{s \rightarrow 0} \frac{L_p(s,V,D)}{s^e e!} = \Ll(V^*,D^*) E^*(V,D) \frac{L(0,M)}{\Omega(V)}.
\end{align*}
Here $\Ll(V^*,D^*)$ is a non-zero number defined in term of the cohomology of the $(\phi,\Gamma)$-module associated with $V$.
\end{conj}
We remark that the conjectures of Bloch and Kato tell us that the aforementioned hypotheses $\mathbf{C1}-\mathbf{C4}$ in \cite{BenLinv2} are a consequence of all the assumptions we have made about $M$. In the case $V$ is ordinary this conjecture has already been stated by Greenberg in \cite{TTT}. In this situation, the $\Ll$-invariant  can be calculated in terms of the Galois cohomology of $V$. 
Conjecturally, the $\Ll$-invariant is non-zero, but even in the cases when $\Ll(V,D)$ has been calculated it is hard to say whether it vanishes or not.\\
We now describe the Galois representation for which we will prove the above-mentioned conjecture.\\
Let $f$ be a modular eigenform of weight $k \geq 2$, level $N$ and Nebertypus $\psi$. Let $K_0$ be the number field generated by the Fourier coefficients of $f$. For each prime $\lambda$ of $K_0$  the existence of a continuous Galois representatation associated to $f$  is well-known
\begin{align*}
\rho_{f,\lambda} : G_{\Q} \rightarrow \mathrm{GL}_2(K_{0,\lambda}).
\end{align*} 
Let $\pfrak$ be the prime above $p$ in $K_0$ induced by the above embedding $\Qb \hookrightarrow \C_p$, we shall write $\rho_{f}:=\rho_{f,\pfrak}$.\\
The adjoint action of $\mathrm{GL}_2$ on the Lie algebra of $\mathrm{SL}_2$ induces a three-dimensional representation of $\mathrm{GL}_2$ which we shall denote by $\mathrm{Ad}$. We shall denote by $\mathrm{Ad}(\rho_{f})$ the $3$-dimensional Galois representation obtained by composing $\mathrm{Ad}$ and $\rho_f$.\\ 
The $L$-function $L(s,\mathrm{Ad}(\rho_f))$ has been studied in \cite{GJ}; unless $f$ has complex multiplication,  $L(s,\mathrm{Ad}(\rho_f))$ satisfies the conjectured functional equation and the Euler factors at primes dividing the conductor of $f$ are well-known. For each $s=2-k, \ldots, 0$, $s$ even, we have that $L(s,\mathrm{Ad}(\rho_f))$ is critical \`a la Deligne and the algebraicity of the special values has been shown in \cite{Stu}.\\

If $p \nmid N$, we choose a $p$-stabilization $\tilde{f}$ of $f$; i.e. a form of level $Np$ such that $f$ and $\tilde{f}$ have the same Hecke eigenvalues outside $p$ and $U_p \tilde{f} = \lambda_p \tilde{f}$, where $\lambda_p$ is one of the roots of the Hecke polynomial at $p$ for $f$. \\
From now on, we shall suppose that $f$ is of level $Np$ and primitive at $N$. We point out that  the choice of a $p$-stabilization of $f$ induces a regular sub-module $D$ of $\boldD_{\mathrm{st}}(\mathrm{Ad}(\rho_f))$. So, from now on, we shall drop the dependence on $D$ in the notation for the $p$-adic $L$-function.\\ 
Following the work of many people \cite{Sc,H6,DD}, the existence of a $p$-adic $L$-function associated to $\mathrm{Ad}(\rho_f)$ when $f$ is ordinary (i.e. $v_p(\lambda_p) =0$) or when $2v_p(\lambda_p) < k-2$  is known.\\
In what follows, we shall not work directly with $\mathrm{Ad}(\rho_f)$ but with $\mathrm{Sym}^2(\rho_f) =\mathrm{Ad}(\rho_f) \otimes \mathrm{det}(\rho_f)$. For each prime $l$, let us denote by $\alpha_l$ and $\beta_l$ the roots of the Hecke polynomial at $l$ associated to $f$. We define
\begin{align*}
D_l(X):=(1-\alpha_l^2X)(1-\alpha_l \beta_l X)(1-\beta_l^2 X).
\end{align*}
For each Dirichlet character $\xi$ we define 
\begin{align*}
\Ll(s,\mathrm{Sym}^2(f),\xi):=(1-\psi^2\xi^2(2)2^{2k-2-2s})\prod_{l}{D_l(\xi(l)l^{-s})}^{-1}.
\end{align*}
This $L$-function differs from $L(s,\mathrm{Sym}^2(\rho_f)\otimes \xi)$ by a finite number of Euler factors at prime dividing $N$ and for the Euler factor at $2$. The advantage of dealing with this {\it imprimitive} $L$-function is that it admits an integral expression (see Section \ref{primLfun}) as the Petersson product of $f$ with a certain product of two half-integral weight forms. The presence of the Euler factor at $2$ in the above definition is due to the fact that forms of half-integral weight are defined only for levels divisible by $4$. This forces us to consider $f$ as a form of level divisible by $4$, thus losing one Euler factor at $2$ if $\psi\xi(2)\neq 0$.\\
Let us suppose that $\lambda_p\neq 0$; then we know that $f$ can be {\it interpolated} in a ``Coleman family''. Indeed, let us denote by $\W$ the weight space. It is a rigid analytic variety over $\Q_p$ such that  $\W(\C_p)=\mathrm{Hom}_{cont}(\Z_p^{\times},\C_p^{\times})$. In \cite{CM},  Coleman and Mazur constructed a rigid-analytic curve $\Cc$  which is locally finite above $\W$ and whose points are in bijection with overconvergent eigenforms.\\
If $f$ is a classical form of non-critical weight (i.e. if $v_p(\lambda_p) < k-1$), then there exists a unique irreducible component of $\Cc$ such that $f$ belongs to it. We fix a neighbourhood $\Cc_F$ of $f$ in this irreducible component, it gives rise to an analytic function $F(\kappa)$ which we shall call a family of eigenforms. Let us denote by $\lambda_p(\kappa)$ the $U_p$-eigenvalue of $F(\kappa)$. We know that $v_p(\lambda_p(\kappa))$ is constant on $\Cc_F$. For any $k$ in $\Z_p$, let us denote by $[k]$ the weight corresponding to $ z \mapsto z^k$. Then for all $\kappa'$ above $[k']$ such that $v_p(\lambda_p(\kappa)) < k-1$ we know that $F(\kappa')$ is classical.\\ Let us fix an even Dirichlet character $\xi$. We fix a generator $u$ of $1+p\Z_p$ and we shall denote by $\lla z \rra $ the projection of $z$ in $\Z_p^{\times}$ to $1+ p\Z_p$. We prove the following theorem in Section \ref{padicL}:
\begin{theo}\label{Tintro}
We have a function $L_p(\kappa,\kappa')$ on $\Cc_{F} \times \W$, meromorphic in the first variable and of logarithmic growth $h=[2 v_p(\lambda_p)]+2$  in the second variable (i.e. $L_p(\kappa,[s])/\prod_{i=0}^h \log_p(u^{s-i}-1)$ is holomorphic on the open unit ball). For any  point $(\kappa, \eps(\lla z \rra)z^s)$ such that $\kappa$ is above $[k]$, $\eps$ is a finite-order character of $1+p\Z_p$ and $s$ is an integer such that $1\leq s \leq k-1$, we have the following interpolation formula
\begin{align*}
 L_p(\kappa,\eps(\lla z \rra)z^s) =  C_{\kappa,\kappa'} E_1(\kappa,\kappa')E_2(\kappa,\kappa') \frac{\Ll(s,\mathrm{Sym}^2(F(\kappa)),\xi^{-1}\eps^{-1}\omega^{1-s})}{\pi^{s}S(F(\kappa))W'(F(\kappa))\lla F^{\circ}(\kappa),F^{\circ}(\kappa)\rra}.
\end{align*}
\end{theo}
Here $E_1(\kappa,\kappa')$ and $E_2(\kappa,\kappa')$ are two Euler-type factors at $p$. We refer to Section \ref{padicL} for the notations. Here we want to point out that this theorem fits perfectly within the framework of $p$-adic $L$-functions for motives and their $p$-adic deformations \cite{CPR,Gpvar,PR}.\\
Our first remark is that such a two variables $p$-adic $L$-function has been constructed in \cite{H6} in the ordinary case and in \cite{Kim} in the non-ordinary case. Its construction is quite classical: first, one constructs a measure interpolating $p$-adically the half-integral weight forms appearing in the integral expression of the $L$-function, and then one applies a $p$-adic version of the Petersson product.\\
Unless $s=1$, the half-integral weight forms involved are not holomorphic but, in Shimura terminology, {\it nearly holomorphic}. It is well known that nearly holomorphic forms can be seen as $p$-adic modular forms (see \cite[\S 5.7]{K2}).\\ In the ordinary case, we have Hida's ordinary projector which is defined on the whole space of $p$-adic modular forms and which allows us to project $p$-adic forms on a finite dimensional vector space where a $p$-adic Petersson product can be defined. \\
If $f$ is not ordinary, the situation is more complicated; $f$ is annihilated by the ordinary projector, and there exists no other projector which could possibly play the role of Hida's projector. The solution is to consider instead of the whole space of $p$-adic forms,  the smaller subspace of {\it overconvergent} ones.\\
On this space $U_p$ acts as a completely continuous operator, and elementary $p$-adic functional analysis allows us to define, for any given $\alpha \in \Q_{\geq 0}$, a projector to the finite dimensional subspace of forms whose slopes with respect to $U_p$ are smaller than $\alpha$.  Then it is easy to construct a $p$-adic analogue of the Petersson product as in \cite{Pan}.\\
The problem in our situation is that nearly holomorphic forms are not overconvergent. Kim's idea is  to construct a space of {\it nearly holomorphic and overconvergent forms} which projects, thanks to a $p$-adic analogue of the holomorphic projector for nearly holomorphic forms, to the space of overconvergent forms. Unfortunately, some of his constructions and proofs are sketched-out, and we prefer to give a new proof of this result using the recent work of Urban.\\
In \cite{UrbNholo}, an algebraic theory for nearly holomorphic forms has been developed; it allows this author to construct a space of {\it nearly overconvergent} forms in which all classical nearly holomorphic forms appear and where one can define an {\it overconvergent} projector to the subspace of overconvergent forms. This is enough to construct, as sketched above, the $p$-adic $L$-function.\\
We expect that the theory of nearly overconvergent forms will be very useful for the construction of $p$-adic $L$-functions; as an example, we can give the generalization of the work of Niklas \cite{Niklas} on values of $p$-adic $L$-function at non-critical integers to finite slope families, or the upcoming work of Eischen, Harris, Li and Skinner on $p$-adic $L$-functions for unitary groups.\\ 
A second remark is that for all weights such that $k > h$ we obtain, by specializing the weight variable, the $p$-adic $L$-functions constructed in \cite{DD}. They construct several distribution $\mu_i$, for $i=1,\ldots, k-1$, satisfying Kummer congruences and the $p$-adic $L$-function is defined via the Mellin transform. The $\mu_i$ define an $h$-admissible measure $\mu$ in the sense of Amice-V\'elu; in this case the Mellin transform is uniquely determined once one knows $\mu_i$ for $i=1,...,h$.\\
If $k \leq h$, then the number of special values is not enough to determine uniquely an analytic one-variable function. Nevertheless, as in Pollack-Stevens \cite{PolSt}, we can construct a well-defined one variable $p$-adic $L$-function for eigenforms such that $k \leq h$ (see Section \ref{padicL}).\\
Let $\kappa_0$ be a point of $\Cc_F$ above $[k_0]$, and $f:=F(\kappa_0)$. We shall write
\begin{align*}
L_p(s,\mathrm{Sym}^2(f),\xi):=L_p(\kappa_0,[s]).
\end{align*}
We now deal with the trivial zeros of this $p$-adic $L$-function. Let $\kappa$ be  above $[k]$ and suppose that $F(\kappa)$ has trivial Nebentypus at $p$, then either $E_1(\kappa,\kappa')$ or $E_2(\kappa,\kappa')$ vanishes when  $\kappa'(u)=u^{k-1}$. The main theorem of the paper is: 
\begin{theo}\label{MainThOC}
Let $f$ be a modular form of trivial Nebentypus, weight $k_0$ and conductor $Np$, $N$ squarefree, even and prime to $p$. Then Conjecture \ref{MainCoOC} (up to the non-vanishing of the $\Ll$-invariant) is true for $L_p(s,\mathrm{Sym}^2(f),\omega^{2-k_0})$.
\end{theo}
In this case, the form $f$ is Steinberg at $p$ and the trivial zero is brought by $E_1$. The proof of this theorem is the natural generalization of the (unpublished) proof of Greenberg and Tilouine in the ordinary case (which has already been generalized to the Hilbert setting in \cite{RosCR,RosH}).\\
We remark that in the forthcoming paper \cite{RosBonn} we remove the hypothesis that $N$ is even and allow $p=2$ when $k_0=2$ using a different method.\\
When we fix $\kappa_0'(u)=u^{k_0}$ we see that $E_1(\kappa,\kappa'_0)$ is an analytic function of $\kappa$. We can then find a factorization $L_p(\kappa,\kappa'_0)=E_1(\kappa,\kappa'_0)L^*_p(\kappa)$, where $L_p^*(\kappa)$ is an {\it improved} $p$-adic $L$-function in the sense of \cite{SSS} (see Section \ref{padicL} for the exact meaning). The construction of the improved $p$-adic $L$-function is similar to \cite{HT}; we substitute the convolution of two measures with the product of a modular form  with an Eisenstein measure. We note that the two-variable $p$-adic $L$-function vanishes on the line $\kappa=[k]$ and $\kappa'=[k-1]$ (the {\it line of trivial zeros}) and we are left to follow the method of \cite{SSS}.\\
The hypotheses on the conductor are to ensure that $\Ll(s,\mathrm{Sym}^2(f))$ coincides with $L(s-k+1,\mathrm{Ad}(\rho_f))$. The same proof gives a proof of Conjecture \ref{MainCoOC} for $\mathrm{Sym}^2(f)\otimes \xi$ for many $\xi$, and $f$ not necessarily of even weight. We refer to Section \ref{Benconj} for a list of such a $\xi$. \\
Recently, Dasgupta \cite{Das} has shown Conjecture \ref{MainCoOC} for all weights in the ordinary case. He uses the strategy outlined in \cite{Citro}. 

\paragraph{Acknowledgement} This paper is part of the author's PhD thesis and we would like to thank J. Tilouine for his constant guidance and attention. We would like to thank \'E. Urban for sharing his ideas on nearly overconvergent forms with the author, and also for inviting him to Columbia University.\\
 We would also like to thank D.~Benois, R.~Brasca, P.~Chojecki, A.~Dabrowski, M.~Dimitrov, R.~Greenberg, F.~Lemma,  A.~Sehanobish, S.~Shah, D.~Turchetti,  S.~Wang, and J.~Welliaveetil for useful conversations.\\
 We thank the anonymous referee for useful comments and suggestions.\\
Part of this work has been written during a stay at the Hausdorff Institute during the {\it Arithmetic and Geometry} program, and a stay at Columbia University. The author would like to thank these institutions for their support and for providing optimal working conditions.

\section{Nearly holomorphic modular forms}
The aim of this section is to recall the theory of nearly holomorphic modular forms from the analytic and geometric point of view, and construct their $p$-adic analogues, the {\it nearly overconvergent} modular forms. We shall use them in Section \ref{padicLfunc} to construct a two variables $p$-adic $L$-function for the symmetric square, generalizing the construction of \cite{Pan}. We want to remark that, contrary to the situation of \cite{Pan}, the theory of nearly overconvergent forms is {\it necessary} for the construction of the two variables $p$-adic $L$-function.\\
 We will also construct an eigenvariety parameterizing finite slope nearly overconvergent eigenforms. The main reference is \cite{UrbNholo}; we are very grateful to Urban for sharing this paper, from its earliest versions, with the author. We point out that there is nothing really new in this section; however, we shall give a proof of all the statements which we shall need in the rest of the paper in the attempt to make it self-contained. We will also emphasize certain aspects of the theory we find particularly interesting. For all the unproven propositions we refer to the aforementioned paper.
\subsection{The analytic definition}\label{analytic}
Nearly-holomorphic forms for $\gl$ have been introduced and extensively studied by Shimura. His definition is of analytic nature, but he succeeded in proving several algebraicity results. Later, Harris \cite{Har1,Har2} studied them in terms of coherent sheaves on Shimura varieties.\\ 
Let $\G$ be a congruence subgroups of $\gl(\Z)$ and $k$ a positive integer. Let $\h$ be the complex upper-half plane,  we let $\gl(\Q)^+$ act on the space of $\Cc^{\infty}$ functions $f:\h \rightarrow \C$ in the usual way
\begin{align*}
f |_{k}\gamma (z)= {\mathrm{det}(\gamma)}^{k/2}{(cz+d)}^{-k}f(\gamma(z))
\end{align*} 
where $\gamma= \left( 
\begin{array}{cc}
a & b \\
c & d
\end{array}
\right) $  and $\gamma(z)=\frac{az+b}{cz+d}$. We now give  the precise definition of nearly holomorphic form.  
\begin{defin}
Let $r\geq 0$ be an integer. Let $f:\h \rightarrow \C$ be a $\Cc^{\infty}$-function, we say that $f$ is a nearly holomorphic form for $\G$ of weight $k$ and degree  $r$ if
\begin{itemize}
\item[i)] for all $\gamma$ in $\G$, we have $f|_k\gamma = f$,
\item[ii)] there are holomorphic $f_i(z)$ with $f_r(z) \neq 0$ such that 
\begin{align*}
f(z)= & \sum_{i=0}^{r} \frac{1}{y^i}f_i(z),
\end{align*}
for $y=\mathrm{Im}{z}$,
\item[iii)] $f$ is finite at the cusps of $\G$.
\end{itemize}
\end{defin}
Let us denote by $\N_k^r(\G,\C)$ the space of nearly holomorphic forms of weight $k$ and degree at most $r$ for $\G$. When $r=0$, we will write $\M_k(\G,\C)$.\\
A simple calculation, as in the case of holomorphic modular forms, tells us  that $k \geq 2r$. \\ 
Finally, let us notice that we can substitute condition ${\it ii)}$ by 
\begin{align*}
\eps^{r+1}(f)= 0
\end{align*}
for $\eps$ the differential operator $-4 y^2 \frac{\partial f}{\partial \overline{z}}$. 
If $f$ belongs to $\N_k^r(\G,\C)$, then $\eps(f)$ belongs to $\N_{k-1}^{r-1}(\G,\C)$. \\
We warn the reader that except for $i=r$, the $f_i$'s are not modular forms. \\
Let us denote by $\delta_{k}$ the  Maa\ss{}-Shimura differential operator  
$$
\begin{array}{cccc}
\delta_k : & \N_k^r(\G,\C) &  \rightarrow & \N_{k+2}^{r+1}(\G,\C)\\
  & f & \mapsto & \frac{1}{2 \pi i}\left( \frac{\partial }{\partial z} + \frac{k}{2 y i} \right) f.
\end{array}
$$

For any integer $s$, we define 
$$
\begin{array}{cccc}
\delta_k^{s} : & \N_k^r(\G,\C) &  \rightarrow & \N_{k+2s}^{r+s}(\G,\C)\\
 & f & \mapsto & \delta_{k+2s-2} \circ \cdots \circ \delta_{k} f.
\end{array}
$$

Let us denote by $E_2(z)$ the nearly holomorphic form 
\begin{align*}
E_2(z) = -\frac{1}{24} +\sum_{n=1}^{\infty} \sigma_1(n)q^{n} + \frac{1}{8 \pi  y}, \:\: \left( \mbox{ where }\forall n \geq 1 \;\; \sigma_1(n)= \sum_{d | n, d >0} d\right).
\end{align*}
It belongs to $\N_2^1(\Sl(\Z),\C)$. It is immediate to see that for any $\G$ and for any form $f \neq 0$ in $\N_2^1(\G,\C)$, it does not exist a nearly holomorphic form $g$ such that $\delta_0 g (z)=f(z) $. This is an exception, as the following proposition, due to Shimura \cite[Lemma 8.2]{ShH3}, tells us.
\begin{prop}\label{sumMS}
Let $f$ in $\N_k^r(\G,\C)$ and suppose that $(k,r)\neq (2,1)$. If $k\neq 2r$, then there exists a sequence $(g_i(z))$, $i=0,\ldots,r$, where $g_i$ is  in $\M_{k-2i}(\G,\C)$ such that 
\begin{align*}
f(z) & = \sum_{i=0}^{r} \delta_{k-2i}^i g_i(z), 
\end{align*}
while if $k=2r$ there exists a sequence $(g_i(z))$, $i=0,\ldots,r-1$,  where $g_i$ is  in $\M_{k-2i}(\G,\C)$  and $c$ in $\C^{\times}$  such that 
\begin{align*}
f(z) & = \sum_{i=0}^{r-1} \delta_{k-2i}^i g_i(z) + c\delta_{2}^{r-1}E_2(z).
\end{align*}
Moreover, such a decomposition is unique.
\end{prop}
The importance of such a decomposition is given by the fact that the various $\delta_{k-2i}^i g_i(z)$ are nearly holomorphic forms. This decomposition will be very useful for the study of the Hecke action on the space $\N_k^r(\G,\C)$.\\
We can define, as in the case of holomorphic modular forms, the Hecke operators as double coset operators. For all $l$ positive integer, we decompose \begin{align*}
\G \left( 
\begin{array}{cc}
1 & 0 \\
0 & l
\end{array}
\right)\G = \cup_i \G \alpha_i .
\end{align*}
We define $f(z)|_k T_l = l^{\frac{k}{2}-1}\sum_{i} f(z)|_k \alpha_i $. We have the following  relations
\begin{align*}
 l \delta_{k}(f(z)|_k T_l) = & (\delta_{k} f ) |_{k+2} T_l, \\
 \eps (f|_k T_l) = & l (\eps f)|_{k-2} T_l.
\end{align*}
\begin{lemma}\label{deltaT_l}
Let $f(z)=\sum_{i=0}^{r} \delta_{k-2i}g_i(z) $ in $\N_k^{r}(\G)$ be an eigenform for $T_l$ of eigenvalue $\lambda_f(l)$,  then $g_i$ is an eigenform for $T_l$ of eigenvalue $l ^{-i}\lambda_f(l)$.
\end{lemma}
\begin{proof}
It is an immediate consequence of the uniqueness of the decomposition in the previous proposition and of the relation between $\delta_k$ and $T_l$.
\end{proof}

Following Urban, we give an alternative construction of nearly holomorphic forms as section of certain coherent sheaves. Such a description will allow us to define a notion of nearly holomorphic forms over any ring $R$. \\
Let $Y=Y(\G)$ be the open modular curve of level $\G$ defined over $\mathrm{Spec}\left(\Z \right)$, and let  $\boldE$ be the universal elliptic curve over $Y$. Let us consider a minimal compactification $X=X(\G)$ of $Y$ and the Kuga-Sato compactification $\overline{\boldE}$ of $\boldE$. Let us denote by $\boldp$ the projection of $\overline{\boldE}$ to $X$ and by $\omega$  the sheaf of invariant differential over $X$, i.e. $\omega=\boldp_{*} \Omega^1_{\overline{\boldE}/X}(\log (\overline{\boldE} / \boldE))$. \\
We define 
\begin{align*}
\h_{\mathrm{dR}}^{1} = R^{1}\boldp_{*} \Omega^{\bullet}_{\overline{\boldE}/X}(\log (\overline{\boldE} / \boldE));
\end{align*}
it is the algebraic de Rham cohomology. 
Let us denote by $\pi: \h \rightarrow \h/\G $ the quotient map, we have over the $\Cc^{\infty}$-topos of $\h$ the splitting
\begin{align*}
\pi^{*}\h_{\mathrm{dR}}^{1} \cong \pi^*\omega \oplus \pi^*\overline{\omega} \cong \pi^*\omega \oplus{\pi^*\omega}^{\vee} .
\end{align*}
Let us  denote by $\pi^{*}\boldE$ the fiber product of $\h$ and $\boldE$ above $Y$. The fiber above $z \in \h$ is the elliptic curve $\C/(\Z+z\Z)$. 
If we denote by $ \tau$ a coordinate on $\C$, the first isomorphism is given in the basis $\textup{d} \tau$, $\textup{d} \overline{\tau}$, while the second isomorphism is induced by the Poincar\'e duality. Let us define  
\begin{align*}
\h_k^r = \omega^{k-r}\otimes \mathrm{Sym}^r(\h_{\mathrm{dR}}^1).
\end{align*}
The above splitting induces 
\begin{align*}
{\h_k^r} \cong \omega^{k} \oplus \omega^{k-2} \oplus \cdots \oplus \omega^{k-2r}.
\end{align*}
We have the Gau\ss{}-Manin connexion 
\begin{align*}
\nabla : \mathrm{Sym}^k(\h_{\mathrm{dR}}^1) \rightarrow \mathrm{Sym}^k(\h_{\mathrm{dR}}^1) \otimes \Omega^1_{X/\Z[N^{-1}]}(\log(\mathrm{Cusp})).
\end{align*}
Recall the descending Hodge filtration on $\mathrm{Sym}^k\left(\h_{\mathrm{dR}}^1\right)$ given by 
\begin{align*}
\mathrm{Fil}^{k-r}\left(\mathrm{Sym}^k\left(\h_{\mathrm{dR}}^1\right)\right) = \h_k^r.
\end{align*}
In particular, we have 
\begin{align*}
0 \rightarrow \omega^k \rightarrow \h_k^r \stackrel{ \tilde{\eps}}{\rightarrow} \h_{k-2}^{r-1} \rightarrow 0.
\end{align*}
By definition, $\nabla$ satisfies Griffiths transversality;
\begin{align*}
\nabla\mathrm{Fil}^{k-r}\left(\mathrm{Sym}^k\left(\h_{\mathrm{dR}}^1\right)\right) \subset \mathrm{Fil}^{k-r-1}\left(\mathrm{Sym}^k\left(\h_{\mathrm{dR}}^1\right)\right)\otimes \Omega^1_{X/\Z[N^{-1}]}(\log(\mathrm{Cusp})) .
\end{align*}
Recall the Kodaira-Spencer isomorphism $\Omega^1_{X/\Z[N^{-1}]}(\log(\mathrm{Cusp})) \cong \omega^{\otimes 2}$; then the map $\nabla$ induces a differential operator  
\begin{align*}
\tilde{\delta}_k : \h_k^r \rightarrow \h_{k+2}^{r+1},
\end{align*}
We have the following proposition \cite[Proposition 2.2.3]{UrbNholo}
\begin{prop}
We have a natural isomorphism $H^{0}\left(X,\h_k^r \right) \cong \N_k^r(\G,\C)$. Once Hecke correspondences are defined on $(X,\h^r_k)$, the above isomorphism is Hecke-equivariant.
\end{prop}
\begin{proof}
Let us denote $R_1 \boldp_{*}\Z = \mathcal{H}\mathpzc{o}\mathpzc{m}(R^1 \boldp_{*}\Z, \Z)$. Via Poincar\'e duality we can identify
\begin{align*}
 \pi^{*}\h_{\mathrm{dR}}^{1} = \mathcal{H}\mathpzc{o}\mathpzc{m}(R_1 \boldp_{*}\Z, \oo_{\h}),
\end{align*}
where $\oo_{\h}$ denote the sheaf of holomorphic functions on $\h$. For all $z \in \h$, we have 
\begin{align*}
 \pi^{*}{(R_1 \boldp_{*}\Z)}_{z} = H_1(\C/(\Z+z\Z),\Z) = \Z+z\Z.
\end{align*}
 Let us denote by $\alpha$ resp. $\beta$ the linear form in $\mathrm{Hom}(R_1 \boldp_{*}\Z, \oo_{\h})$ which at the stalk at $z$ sends $a+bz$ to $a $ resp. $b$. It is the dual basis of $\gamma_1$, $\gamma_2$. Let $\eta \in H^{0}(X,\h_k^r)$, we can write 
\begin{align*}
 \pi^* \eta = \sum_{i=0}^r f_i(z){\textup{d} \tau}^{{\otimes}^{k-i}}{\beta}^{{\otimes}^{i}}.
\end{align*}
We remark that we have $\beta = \frac{\textup{d}\tau - \textup{d}\overline{\tau}}{2iy}$. 
Using the Hodge decomposition of ${\h_k^r}_{/\Cc^{\infty}}$ given above, we project $ \pi^* \eta$ onto $\omega^{{\otimes}^k}$ and we obtain 
\begin{align*}
 f(z) = \sum_{i=0}^r \frac{f_i(z)}{(2 i y)^i}.
\end{align*}
It is easy to check that $f(z)$ belongs to $\N_k^r(\G,\C)$ and that such a map is bijective.
\end{proof}

\begin{rem}
This proposition allows us to identify $\tilde{\eps}$ with the differential operator $\eps$ and $\tilde{\delta}_k$ with the Maa\ss{}-Shimura operator $\delta_k$.
\end{rem}

We now give  another description of the sheaf $\h_k^r$; for any ring $R$ we shall denote by $R[X]
_r$ the group of polynomial with coefficients in $R$  of degree at most $r$. Let us denote by $B$ the standard Borel of upper triangular matrices of $\Sl$.  We have a left action of $B(R)$ over $\Aa^1(R) \subset \mathbb{P}^1(R)$  via the usual fractional linear transformations
\begin{align*}
{\left( \begin{array}{cc}
a & b \\
0 & a^{-1}\end{array}
\right)}.X = & \frac{aX+b}{a^{-1}}.
\end{align*}
We define then a right action of weight $k \geq 0$ of $B(R)$ on $R[X]_r$ as
\begin{align*}
 P(X)\vert_k \left( 
\begin{array}{cc}
a & b \\
0 & a^{-1}
\end{array}
\right)= & a^k P\left(a^{-2}X+ba^{-1}\right).
\end{align*}
If we see $P(X)$ as a function on  $\Aa^1(R)$, then 
\begin{align*}
 P(X) \vert_k \gamma = a^k P\left(\left( 
\begin{array}{cc}
a^{-1} & b \\
0      & a
\end{array}
\right). X \right).
\end{align*}
We will denote by $R[X]_r(k)$ the group $R[X]_r$ endowed with this action of the Borel. We now use  this representation of $B$ to give another description of  $\h_k^r$.\\ 
We can define a $B$-torsor $\Tt$ over $Y_{Zar}$ which consists of isomorphism $\psi_U: \h^1_{DR/U} \cong \oo(U) \oplus \oo(U) $ such that on the first component it induces $\oo(U) \cong\omega_{/U} $ and on the quotient it induces $\oo(U)\cong \omega^{\vee}_{/U}$, for $U$ a Zariski open of $Y$. That is, $\Tt$ is the set of trivialization of $\h^{1}_{\mathrm{dR}}$ which preserves the line spanned by a fixed invariant differential $\omega$ and the Poincar\'e pairing. We have a right action of $B$ on such a trivialization given by 
\begin{align*}
(\omega,\omega' ) \left( \begin{array}{cc}
a & b \\
0 & a^{-1}
\end{array}
\right) = (a \omega, a^{-1}\omega' +b \omega).
\end{align*}
We can define similarly an action of the Borel of $\gl(R)$ but this would not  respect the Poincar\'e pairing.\\
Then we define the product $\Tt \times ^B R[X]_r(k)$, consisting of couples $(t,P(X))$ modulo the relation $(t \gamma, P(X))\sim (t,  P(X)\vert_k \gamma^{-1})$, for $\gamma$ in $B$. It is isomorphic to $\h_k^r$ as $R$-sheaf over $Y$. In fact, a nearly holomorphic modular form can be seen as a function 
\begin{align*}
 f : \Tt \rightarrow  R[X]_r(k)
\end{align*}
which is $B$-equivariant. That is, $f$ associates to an element $(E,\mu,\omega, \omega')$ in $\Tt$ ($\mu$ denotes a level structure) an element 
$f(E,\mu,\omega, \omega')(X)$ in $R[X]_r$ such that 
\begin{align*}
f(E,\mu, a\omega, a^{-1}\omega' + b \omega) = & a^{-k}f(E,\mu,\omega, \omega')(a^2 X -ba). 
\end{align*}

We are now ready to introduce a {\it polynomial} $q$-expansion principle for nearly holomorphic forms. Let us pose  $A=\Z\left[ \frac{1}{N}\right]$; let $\mathrm{Tate}(q)$ be the Tate curve over $A[[q]]$,  $\omega_{can}$ the canonical differential and   $\alpha_{can}$ the canonical $N$-level structure. \\
We can construct as before (take $\mathrm{Tate}(q)$ and $A[[q]]$ in place of $\tilde{\boldE}$ and $X$) the Gau\ss{}-Manin connection  $\nabla$, followed by the contraction associated to the vector field $ q \frac{\textup{d}}{\textup{d}q}$
\begin{align*}
\nabla\left( q \frac{\textup{d}}{\textup{d}q}\right): \h_{\mathrm{dR}}^{1}(\mathrm{Tate}(q)_{/ A((q))}) \rightarrow \h_{\mathrm{dR}}^{1}(\mathrm{Tate}(q)_{/A((q))}).
\end{align*}
We pose $u_{\mathrm{can}} := \nabla\left( q \frac{\textup{d}}{\textup{d}q} \right)(\omega_{\mathrm{can}})$. We remark that $(\omega_{\mathrm{can}},u_{\mathrm{can}})$ is a basis of $\h_{\mathrm{dR}}^{1}(\mathrm{Tate}(q)_{/A((q))})$ and that $u_{\mathrm{can}}$ is horizontal for the Gau\ss{}-Manin connection (moreover $u_{\mathrm{can}}$ is a basis for the {\it unit root subspace}, defined by Dwork, which we will describe later). For any $A$-algebra $R$ and $f$ in $\N_k^r(\G,R)$, we say that 
\begin{align*}
f(q,X) := f(\mathrm{Tate}(q),\mu_{\mathrm{can}},\omega_{\mathrm{can}},u_{\mathrm{can}})(X)  \in R[[q]][X]
\end{align*} 
is the {\it polynomial} $q$-expansion of $f$. If we take a form $f$ in $\N_k^r(\G,\C)$ written in the form $\sum_i^r f_i(z)\frac{1}{(-4 \pi y)^i}$
we obtain 
\begin{align*}
 f(q,X)= \sum_i^r f_i(q) X^i \;\;\:\: (i.e. \; X ``=" -\frac{1}{4 \pi y} ).
\end{align*}
For example, we have 
\begin{align*}
 E_2(q,X) = -\frac{1}{24} +\sum_{n=1}^{\infty} \sigma_1(n)q^{n} - \frac{X}{2}, 
\end{align*}
We have the following proposition \cite[Proposition 2.3]{UrbNholo}
\begin{prop}
Let $f$ be in $\N_k^r(\G,R)$ and let $\eps (f)$ in $\N_{r-1}^{k-2}(\G,R)$. Then for all $(E,\mu,\omega, \omega')$ in $\Tt$  we have
\begin{align*}
\eps(f) (E,\mu,\omega, \omega') = & \frac{\textup{d}}{\textup{d}X} f (E,\mu,\omega, \omega')(X).
\end{align*}
\end{prop}
Note that if $r!$ is not invertible in $R$, then $\eps$ is NOT surjective.
We have that $E_2$ is defined over $\Z_p$ for $p\geq 5$. As $\eps  \left( 2 E_2(q,X) \right)=-1$ we have that $- 2 E_2(q,X)$ gives a section for the map $\eps : \h^1_{\mathrm{dR}} \rightarrow \omega$.

\begin{rem}
There is also a representation-theoretic interpretation of $y \delta_k$ in term of Lie operator and representation of the Lie algebra of $\Sl(\R)$. More details can be found in $\cite[\S 2.1, 2.2]{Bump}$.
\end{rem}

\subsection{Nearly overconvergent forms}\label{overconvergen}
In this section we give the notion of nearly overconvergent modular forms \`a la Urban. Let $N$ be a positive integer and $p$ a prime number coprime with $N$. Let $X$ be $X(\G)$ for $\G=\G_1(N)\cap \G_0(p)$ and let $X_{\mathrm{rig}}$ the generic fiber of the associated formal scheme over $\Z_p$. Let $A$ be a lifting of the Hasse invariant in characteristic $0$. If $p\geq 5$, we can take $A=E_{p-1}$, the Eisenstein series of weight $p-1$. For all $v$ in ${\Q}$ such that  $v \in [0, \frac{p}{p+1}]$ we define $X_N(v)$ as the set of $x$ in $X(\G_1(N))_{rig}$ such that $|A(x)|\geq p^{-v}$. The assumption that  $ v \leq \frac{p}{p+1}$ is necessary to ensure the existence of the canonical subgroup of level $p$. Consequently, we have that $X_N(v)$ can be seen as an affinoid of $X_{\mathrm{rig}}$ via the map 
\begin{align*}
u : (E,\mu_N) \mapsto (E,\mu_N, C),
\end{align*} where $C$ is the canonical subgroup. Let us define $X(v):= u (X_N(v))$. We define  $X_{\mathrm{ord}}$ as the ordinary multiplicative locus of $X_{\mathrm{rig}}$, i.e. $X_{\mathrm{ord}}=X(0)$. For all $v$ as above,  $X(v)$ is a rigid variety and a strict neighborhood of $X_{\mathrm{ord}}$.\\
We remark that the set of $x$ in $X_{\mathrm{rig}}$ such that $|A(x)|\geq p^{-v}$ consists of two disjoint connected components, isomorphic via the Fricke involution, and that $X(v)$ is the connected component containing $\infty$.
We define, following Katz \cite{Katz}, the space  of $p$-adic modular forms of weight $k$ as 
\begin{align*}
\M_k^{p-\mathrm{adic}}(N)= & H^{0}(X_{\mathrm{ord}},\omega^{{\otimes}^k}).
\end{align*} 
We say that a $p$-adic modular form $f$ is {\it overconvergent} if $f$ can be extended to a strict neighborhood of $X_{\mathrm{ord}}$. That is, there exists $v > 0$ such that $f$ belongs to $H^{0}(X(v),\omega^{{\otimes}^k})$. Let us define the space of overconvergent modular forms 
\begin{align}\label{defover}
\M_k^{\dagger}(N)= & \varinjlim_{v >0} H^{0}(X(v),\omega^{{\otimes}^k}).
\end{align}
In the same way, we define the set of {\it nearly overconvergent} modular forms as 
\begin{align*}
\N_k^{r,\dagger}(N)= \varinjlim_{v >0} H^{0}(X(v),\h_k^r).
\end{align*}
The sheaf $\h_k^r$ is locally free as $E_2(q,X)$ gives a splitting of $\h^1_{\mathrm{dR}}$ and we can consequently find an isomorphism  $H^{0}(X(v),\h_k^r) \cong \oo(X(v))^M$. For $v' < v$,  these isomorphisms are compatible with the restiction maps $X(v) \rightarrow X(v')$. The supremum norm on $X(v)$ induces a norm on each $H^{0}(X(v),\h_k^r)$ which makes this space a Banach module over $\Q_p$. This allows moreover to define an integral structure on $H^{0}(X(v),\h_k^r)$. For all $\Z_p$-algebra $R$, we shall denote by $\M_k^{\dagger}(N,R)$, $\N_k^{r,\dagger}(N,R)$ the global section of the previous sheaves when they are seen as sheaves over $X(v)_{/R}$.\\
We have a correspondence 
 $$
 \xymatrix { 
 & C_p\ar[dl]_{p_1} \ar[dr]^{p_2}  \\
 X(v) &  & X(v) .}
 $$
On the non-compactified modular curve, over $\Q_p$, $C_p$ is the rigid curve classifying quadruplets $(E,\mu_N, C, H)$ with $|A(E)|\geq p^{-v}$, $\mu_N$ a $\G_1(N)$-structure, $C$ the canonical subgroup and $H$ a subgroup of $E[p]$ which intersects $C$ trivially. The projections are explicitly given by 
\begin{align*}
p_1 (E,\mu_N, C, H) = &  (E,\mu_N, C),\\
p_2 (E,\mu_N, C, H) = &  (E/H,Im(\mu_N), E[p]/H).
\end{align*}
We remark that the theory of canonical subgroups ensures us that if $v\leq \frac{1}{p+1}$ then $E[p]/H$ is the canonical subgroup of $E/H$ (and the image of $C$ modulo $H$, of course). The map $p_2$ induces an isomorphism $C_p \cong X\left(\frac{v}{p}\right)$.\\
We define the operator  $U_p$ on $H^{0}(X(v),\h_k^r)$ as the following map  
\begin{align*}
H^{0}(X(v),\h_k^r) \rightarrow H^{0}(X(v),p_2^*\h_k^r) \stackrel{p^{-1}\mathrm{Trace}(p_1)}{\rightarrow} H^{0}(X(v),\h_k^r).
\end{align*}
The fact that $p_2$ is an isomorphism implies the well known property that {\it $U_p$ improves overconvergence}.\\
We can construct correspondences as in \cite[\S 4]{Pil} to define operators $T_l$ for $l\nmid Np$ and $U_l$ for $l \mid N$.\\
Let $A$ be a Banach ring, and let $U:M_1 \rightarrow M_2$ be a continuous morphism of $A$-Banach modules. We pose
\begin{align*}
|U |= \sup_{m \neq 0} \frac{|U(m)|}{|m|}.
\end{align*}
This norm induces a topology on the module of continuous morphisms of $A$-Banach modules. 
 We say that an operator $U$  is {\it of finite rank} if it is a continuous morphism of $A$-Banach modules such that its image is of finite rank over $A$. We say that $U$ is {\it completely continuous} if it is a limit of finite rank operators. Completely continous operators admit a Fredholm determinant \cite[Proposition 7]{SerreCC}.\\ 
We give to $H^{0}(X(v),\omega^{{\otimes}^k})$  the structure of Banach space for the norm induced by the supremum norm on $X(v)$; the transition maps in \ref{defover} are completely continuous and we complete $M_k^{\dagger}(N)$ for this norm. It is known that $U_p$ acts as a completely continuous operator  on this completion; its Fredholm determinant is independent of $v$, for $v$ big enough \cite[Theorem B]{Col}. 
Similarly, we have that $U_p$ is completely continuous  on $\N_k^{r,\dagger}(N)$. Indeed, $U_p$ is the composition of the restriction to $X\left(\frac{v}{p}\right)$ and a trace map.\\
 On $q$-expansion, $U_p$ amounts to 
\begin{align*}
 \sum_{i=0}^r \sum_n a_n^{(i)}q^n X^i \mapsto \sum_{i=0}^r \sum_n a_{pn}^{(i)}q^n p^i X^i.
\end{align*}
We now recall  that we have on ${\h^1_{\mathrm{dR}}}_{ / X_{\mathrm{ord}}}$ a splitting $\omega \oplus U$. Here $U$ is a Frobenius stable line  where the Frobenius is invertible. Some authors call this splitting  the {\it unit root splitting}. It induces  ${\h_k^r}_{/X_{\mathrm{ord}}}=\omega^{{\otimes}^k}\oplus\cdots \oplus U^{{\otimes}^r}\otimes \omega^{{\otimes}^{k-r}}$. We have then \cite[Proposition 3.2.4]{UrbNholo}
\begin{prop}\label{ordproj}
The morphism
$$
\begin{array}{ccccc}
H^{0}(X(v),\h_k^r) & \rightarrow & H^{0}(X_{\mathrm{ord}},\h_k^r) & \rightarrow & H^{0}(X_{\mathrm{ord}},\omega^{{\otimes}^k}) \\
 f(X) & \mapsto & f(X)_{{|}_{X_{\mathrm{ord}}}} & \mapsto & f(0)
\end{array}$$ is injective and commutes with $q$-expansion.
\end{prop}
Note that the injectivity of the composition is a remarkable result. 
A consequence  of this is that every nearly overconvergent form has a unique degree $r$ \cite[Corollary 3.2.5]{UrbNholo}.\\
We remark that we have two differential maps 
$$\begin{array}{cccc}
 \eps :& \N_k^{r,\dagger}(\G) & \rightarrow & \N_{k-2}^{r-1,\dagger}(\G),\\
 \delta_k :& \N_k^{r,\dagger}(\G)&  \rightarrow & \N_{k+2}^{r+1,\dagger}(\G).
\end{array}
$$
Both of them are induced by functoriality from the maps defined in Section \ref{analytic} at the level of sheaves. 
We want to mention that Cameron in his PhD thesis \cite[Definition 4.3.6]{Cameron} gives an analogue of the Maa\ss{}-Shimura  differential operator for rigid analytic modular forms on the Cerednik-Drinfeld $p$-adic upper half plane. It would be interesting to compare his definition with this one.\\
The above mentioned splitting allows us to define a map 
$$
\begin{array}{ccccc}
\Theta : \M_k^{\dagger}(N) & \stackrel{\delta_k}{\rightarrow} & \N_{k+2}^{1,\dagger}(N) & \rightarrow & \M_{k+2}^{p-\mathrm{adic}}(N) 
\end{array}$$ 
which at level of $q$-expansion is $q\frac{\textup{d}}{\textup{d} q}$. We have the following application of Proposition; \ref{ordproj}
\begin{coro}
Let $f$ be an overconvergent form of weight different from $0$, then $\Theta f$ is not overconvergent.
\end{coro}
We have the following proposition \cite[Lemma 3.3.4]{UrbNholo};
\begin{prop}\label{sumMSover}
 Let $(k,r) $ be different from $(2,1)$ and $f$ in $\N_k^{r,\dagger}(N,R)$.  If $k\neq 2r$, then there exist $g_i$, $i=0,\ldots,r$,  in $M^{\dagger}_{k-2i}(N,R)$ such that 
\begin{align*}
f & = \sum_{i=0}^{r} \delta_{k-2i}g_i, 
\end{align*}
while if $r=2k$ there exists a sequence $(g_i)$, $i=0,\ldots,r-1$, with each $g_i$ in $M^{\dagger}_{k-2i}(N,R)$  and $c$ in $R$  such that 
\begin{align*}
f & = \sum_{i=0}^{r-1} \delta_{k-2i}g_i + c\delta_{2}^{r-1}E_2.
\end{align*}
Moreover, such a decomposition is unique.
\end{prop}
We conclude with a sufficient condition for a nearly overconvergent modular form to be classical;
\begin{prop}
 Let $k$ be a classical weight, $f$ in $\N_k^{r,\dagger}(N)$ an eigenform for $U_p$ of slope $\alpha$. Then $r\leq \alpha$. If $\alpha < k-1 +r$, then $f$ is classical.
\end{prop}
\begin{proof}
 The first  part is a trivial consequence of the above formula for $U_p$ acting on $q$-expansion. For the second part, the hypotheses of Proposition \ref{sumMSover} are satisfied. We apply $\eps^r$ to $f$ to see that $g_r$ is of degree $0$, slope $\alpha - r$ and weight $k-2r$. It is then known that $g_r$ is classical. We conclude by induction on the degree.
\end{proof}
Let $\alpha \in \Q_{\geq 0}$ and $r \geq 0$ be a positive integer such that $r\leq \alpha$. We say that a positive integer $k$ is a {\it non critical weight} with respect to $\alpha$ and $r$ if  $\alpha < k-1 +r$.\\
\begin{rem}
In particular, if $\alpha=0$ then $r=0$. This should convince the reader of the fact that the ordinary projector is a $p$-adic analogue of the holomorphic projector.
\end{rem}

\subsection{Families}\label{families}
In this subsection we construct families of nearly overconvergent forms. 
We start recalling the construction of families of overconvergent modular forms as done in Andreatta-Iovita-Stevens \cite{AIS} and Pilloni \cite{Pil}.\\ 
The authors of the  first paper use  $p$-adic Hodge theory to construct their families, while Pilloni's approach is more in the spirit of Hida's theory. We will follow in our exposition  the article \cite{Pil}.\\
Let us denote by $\W$ the weight space. It is a rigid analytic variety over $\Q_p$ such that $\W(\C_p)=\mathrm{Hom}_{\mathrm{cont}}(\Z_p^{\times},\C_p^{\times})$. For all integer $k$, we denote the continuous homomorphism $z \mapsto z^{k}$ by $[k]$.\\ 
 Let $\Delta=\mu_{p-1}$ if $p >2$ (resp. $\Delta=\mu_2$ if $p=2$) and let $B(1,1^{-})$ be the open unit ball centered in $1$. It is known that $\W$ is an analytic space isomorphic to $\Delta \times B(1,1^{-})$; let us denote by $\A(\W)$ the ring  of analytic function on $\W$. We define for $t$ in $(0,\infty)$,
 \begin{align*}
 \W(t) := & \left\{ (\zeta, z) \in \W(\C_p) | |z-1| \leq p^{-t} \right\}.
 \end{align*} 
Let $\Delta$ be the cyclic group of $q$-roots of unity. We define $\kappa$,  the universal weight, as
$$
\begin{array}{cccc}
\kappa : & \Z_p^{\times} & \rightarrow & {(\Z_p[\Delta][[S]])}^{\times} \\
 & (\zeta, z) & \mapsto & \tilde{\zeta} (1+S)^{\frac{\log_p(z)}{\log_p(u)}},
\end{array}
$$
where $\tilde{\zeta}$ is the image of $\zeta$ via the tautological character $\Delta \rightarrow {(\Z_p[\Delta])}^{\times}$.
We can see $\kappa$ as a local coordinate on  the open unit ball $\lgr 1 \rgr \times B(1,1^{-})$.\\
For any weight $\kappa_0$, both of the aforementioned papers construct an invertible sheaf $\omega^{\kappa_0}$ over $X(v)$  whose sections correspond to overconvergent forms of weight $\kappa_0$. This construction can be globalized over $\W(t)$ into a coherent sheaf  $\omega^{\kappa}$ over $X(v)\times \W(t)$ (for suitable $v$ and $t$) such that the corresponding sections will give rise to families of holomorphic modular forms.\\
We describe more in detail Pilloni's construction. Let $n,v$ be such that $0 \leq v < {\frac{1}{p^{n-2}(p+1)}}$; there exists then a {\it canonical subgroup}  $H_n$ of level $n$ over $X(v)$. It is possible to define a rigid variety $F^{\times}_n(v)$ above $X(v) $ whose  $\C_p$-points are triplets $(x,y;\omega)$ where $x$ is an element of $X(v)$ corresponding to an elliptic curve $E_x$, $y$ a generator of $H_n^D$ (the Cartier dual of $H_n$) and $\omega$ is an element of $e^{\ast}\Omega_{E_x/\C_p}$ (for $e$ the unit section  $X(v)\rightarrow E_x$) whose restriction to $e^{\ast}\Omega_{H_n/\C_p}$ is the image of $y$ via the Hodge-Tate map \cite[\S 3.3]{Pil}. Locally, $F_n^{\times}(v)$ is a trivial fibration of $X(v)$ in $p^{n-1}(p-1)$ balls.\\ 
On $F_n(v)^{\times}$ we  have an action of ${(\Z/p^{n})}^{\times}$. This induces an action of $\Z_p^{\times}$.
For each $t$, there exist  $v$ and $n$ satisfying the above condition such that any $\kappa_0$ in $\W(t)$ acts on $F_n^{\times}(v)$. Let us denote by $\pi_n(v)$ the projection from $F_n^{\times}(v)$ to $X(v)$; $\omega^{\kappa_0}$ is by definition the $\kappa_0$-eigenspace of $\left({\pi_n(v)}_*\oo_{F_n^{\times}(v)}\right)$ (which we shall denote by $\left({\pi_n(v)}_*\oo_{F_n^{\times}(v)}\right) \lla \kappa_0 \rra$) for the action of $\Z_p^{\times}$. If $k$ is a positive integer, then $\omega^{[k]} = \omega^{\otimes k}$, for $\omega$ the sheaf defined in Section \ref{analytic}.\\
A family of overconvergent modular forms is then an element of 
\begin{align*}\M(N,\A(\W(t))) := & \varinjlim_v H^0\left(X(v) \times \W(t), \omega^{\kappa} \right),\\
 \omega^{\kappa} = & \left({\pi_n(v)}_*(\oo_{F^{\times}_n(v)}\hat{\otimes} \oo_{\W})  \right)\lla \kappa \rra. \end{align*} 
The construction commutes to base change in the sense that for weights $\kappa_0 \in \W(t)(K)$ we have 
\begin{align*}
\omega^{\kappa}\otimes_{\kappa_0}K = \omega^{\kappa_0}_{/K}. 
\end{align*}

The operator $U_p$ defined in the previous section is completely continuous on $\M(N,\A(\W(t)))$.  Let $Q_0(\kappa,T)$ be its Fredholm determinant; it is independent of $v$ and belongs to $\Z_p[[\kappa]][[T]] $ \cite[Theorem 4.3.1]{CM}.\\
This definition includes the family of overconvergent modular forms \`a la Coleman. Let $\zeta^*(\kappa)$ be the $p$-adic $\zeta$-function, we pose 
\begin{align}\label{GEis}
\tilde{E}(\kappa)= \frac{\zeta^*(\kappa)}{2} + \sum_{n}\sigma^*_n(\kappa) q^n,
\end{align}
where $\sigma^*_n(\kappa) = \sum_{1 \leq d | n, (d,p)=1} \kappa(d)d^{-1}$, and 
\begin{align*}
 E(\kappa) = \frac{2}{\zeta^*(\kappa)} \tilde{E}(\kappa).
\end{align*}
It is known that the zeros of $E(\kappa)$ are {\it far enough} from the ordinary locus \cite[B1]{Col}, in particular there exists $v$ such that $ E(\kappa)$ is invertible on $X(v)\times \W(t)$. In \cite[B4]{Col} a family of modular forms $F(\kappa)$ is defined as an element of $\A(\W(t))[[q]]$ such that for all $\kappa \in \W(t)$, we have $\frac{F(\kappa)}{E(\kappa)}$ in $H^0(X(v)\times \W(t), \oo_{X(v) \times \W})$. The fact that $E(\kappa)$ is invertible induces an isomorphism 
\begin{align*}
 H^0(X(v)\times \W(t), \oo_{X(v) \times \W}) \stackrel{ \times E(\kappa)}{\longrightarrow} H^0\left(X(v) \times \W(t), \omega^{\kappa} \right).
\end{align*}
Let us define the following coherent sheaf
\begin{align*}
\h_{\kappa}^{r}= & \omega^{\kappa[-r]} \otimes \mathrm{Sym}^r(\h_{\mathrm{dR}}^1);
\end{align*}
we define then for all affinoid $\U \subset \W(t)$ the family of nearly overconvergent forms of degree $r$ 
\begin{align*}
\N^{r}(N,\A(\U))= & \varinjlim_{v} H^{0}(X(v)\times \U,\h_{\kappa}^{r}\hat{\otimes} \oo_{\U}).
\end{align*} 
We remark that we can choose $v$ small enough such that $H^{0}(X(v)\times \U,\omega^{[-r]} \otimes \mathrm{Sym}^r(\h_{\mathrm{dR}}^1)\hat{ \otimes} \oo(\U))$ is isomorphic via multiplication by $E(\kappa)$ to $H^{0}(X(v)\times \U,\h_{\kappa}^{r}\otimes \oo(\U))$. We shall call the elements of the former space families of nearly overconvergent forms {\it \`a la Coleman}. \\
We can define $\N^{\infty}(N,\A(\U))$ as the completion of $\cup_r\N^{r}(N,\A(\U))$ with respect to the Frechet topology. For the interested reader, let us mention that there exist forms in $\N^{\infty}(N,\A(\U))$ whose polynomial $q$-expansion is no longer a polynomial in $X$ but an effective formal series. Indeed, we can trivialize ${\h_{\kappa}^{r}}_{/X(v)\times \W(t)}$ as $\oplus_{i=0}^{r} \omega^{\kappa[-2i]}$ and take a sequence of $f_r=(f_{r,0},\ldots,f_{r,r})$, $f_r$ in  $\N^{r}(N,\A(\U))$, such that $f_{r,i}=f_{r+1,i}$ and $f_{r,r}$  smaller and smaller for the norm induced by $X(v)$.\\
There is a sheaf-theoretic interpretation of $\N^{\infty}(N,\A(\U))$. Let $\A\mathpzc{n}(\Z_p)$ be the ring of analytic function on $\Z_p$ with values in $\A(\U)$; we can define the vector bundle in Frechet space 
\begin{align*}
\h_{\kappa}^{\infty} = \Tt \times^{B}\A\mathpzc{n}(\Z_p).
\end{align*}
\begin{rem}
As in the rest of the paper we will work with nearly overconvergent forms of bounded slope, there is no particular interest in taking $r=\infty$ as we have already mentioned the degree gives a lower bound on the slopes which can appear. However, we think that the case $r=\infty$ could have some interesting applications, both geometric or representation-theoretic.
\end{rem}
We can see that $U_p$ acts completely continuously on $\N_{\kappa}^{r}(N,\A(\W(t)))$ using \cite[Proposition A5.2]{Col}, as it is defined via the correspondence $C_p$. We have on ${\N^{r}(N,\A(\W(t)))}$ an action of the Hecke algebra $\T^r(N,\A(\W(t)))$ generated by the Hecke operators $T_l$, for $l$ coprime with $Np$, and $U_l$ for $l$ dividing $Np$. We will denote by $Q_r(\kappa,T)$ the Fredholm determinant of $U_p$ on $\N_{\kappa}^{r}(\G,\A(\W(t)))$. To lighten the notation, we will write sometimes $Q_r(T)$ for $Q_r(\kappa,T)$ if there is no possibility of confusion.\\ 
\begin{lemma}
For any $t \in (0,\infty)$ and suitable $v$ (see \cite[\S 5.1]{Pil}) small enough and $t$ big enough, $H^{0}(X(v)\times \W(t),\h_{\kappa}^{r}\otimes \oo_{\W})$ is a direct factor of a potentially orthonormalizable $\A(\W(t))$-module (see \cite[page 7]{Buz} for the definition of potentially orthonormalizable).
\end{lemma}
\begin{proof}
The proof is exactly the same as  \cite[Corollary 5.2]{Pil}, so we only sketch it. Let 
\begin{align*}
M:= H^{0}(X(v)\times \W(t),\h_{\kappa}^{r}\otimes \oo_{\W}), \;\;\:\:\;
A:= \A(\W(t)).
\end{align*}   Let us denote by $B$ the function ring of $X(v)$ and by $B'$ the function ring of ${(H_n^{D})}^{\times}$ above $X(v)$. We know that $B'$ is an \'etale $B$-algebra of Galois group ${(\Z/p^n\Z)}^{\times}$. As $M$ is a direct summand of  $M'=M\otimes_B B'$,  it will be enough to show that the latter is potentially orthonormalizable. Let ${(\U_i)}_{i=1,\ldots,I} \rightarrow {(H_n^{D})}^{\times}$ be a finite cover by open sets such that for all $i$'s $F_n^{\times}(v)\times_{X(v)}\U_i$ is a disjoint union of $p^{n-1}(p-1)$ copies of $\U_i$.
The augmented \v{C}ech complex associated to this cover is then 
\begin{align*}
0 \rightarrow M' \rightarrow M_{1} \rightarrow \cdots  M_I \rightarrow 0
\end{align*}
and it is exact. Let $k \geq 1$ be an integer  and $\underline{i}$ be a subset of $\left\{1,2,\ldots, I \right\}$  of cardinality $k$. By construction $M_k$ is a sum of modules of the type $M' \hat{\otimes}_{B'} B_{\underline{i}}$ for $B_{\underline{i}}={\hat{\otimes}_{j \in \underline{i}}}\oo(\U_{j})$ where the tensor product is taken over $B'$. By the choice of $\U_i$, each one of these modules  is free of rank $r+1$ over $A \hat{\otimes} B_{\underline{i}}$. As $B_{\underline{i}}$ is potentially orthonormalizable over $\Q_p$ we know that $A \hat{\otimes} B_{\underline{i}}$ is potentially orthonomalizable over $A$. We can conclude by \cite[Lemma 5.1]{Pil}.\\ 
\end{proof}
We can thus apply Buzzard's eigenvariety machinery \cite[Construction 5.7]{Buz}. This means that to the data 
\begin{align*}
 (\A(\W(t)),{\N^{r}(N,\A(\W(t)))},\T^r(N,\A(\W(t))),U_p)
\end{align*}
we can associate a rigid-analytic one-dimensional variety $\Cc^r(t)$. Let us denote by $Z$ the zero-locus of $Q_r(T)$ on $\W(t) \times \Aa^1_{\mathrm{An}}$ (see  \cite[\S 3.4]{UrbNholo}). The  rigid-analytic variety $\Cc^r(t)$ is characterized by the following properties.
\begin{itemize}
 \item We have a finite map $\Cc^r(t) \rightarrow Z$.
 \item There is a cover of $Z$ by affinoid $Y_i$ such that $X_i = Y_i \times_Z \W(t)$ is an open affinoid of $\W(t)$ and $Y_i \rightarrow X_i$ is finite. 
 \item Above $X_i$ we can write $Q_r(T)=R_r(T)S_r(T)$ with $R_r(T)$ a polynomial in $T$ whose constant term is $1$ and $S_r(T)$ power series in $T$ coprime to $R_r(T)$. 
\item Let $R^*_r(T)=T^{\mathrm{deg}(R_r(T))} R_r(T^{-1})$. Above $X_i$ we have a $U_p$-invariant decomposition \begin{align*}
 {\N^{r}(N,\A(X_i))}={\N^{r}(N,\A(X_i))}^{*} \bigoplus {\N^{r}(N,\A(X_i))}',
\end{align*} such that $R^*_r(U_p)$ acts  on  ${\N^{r}(N,\A(X_i))}^{'}$ invertibly and on ${\N^{r}(N,\A(X_i))}^*$ is $0$. Moreover the rank of ${\N^{r}(N,\A(X_i))}^{*}$ on $\A(X_i)$ is $\mathrm{deg}(R_r(T))$.
\item There exists a coherent sheaf $\widetilde{\N^{r}(N,\A(\W(t)))}$ above $\Cc^r(t)$.
\item To each $K$-point $x$ of $\Cc^r(t) \times_Z Y_i$ above $\kappa(x) \in \W(t)$ corresponds a system of Hecke eigenvalues for $\T_{\kappa}^r(N,K)$ on $\N_{\kappa(x)}^{r,\dagger}(Np,K)$ such that the $U_p$-eigenvalue is a zero of $R^*_r(T)$ (in particular it is not zero).
\item To each $K$-point $x$ as above, the fiber $\widetilde{\N^{r}(N,\A(\W(t)))}_x$ is the generalized eigenspace in $\N_{\kappa(x)}^{r,\dagger}(Np,K)$ for the system of eigenvalues associated to $x$.
\end{itemize}
Taking the limit for $t$ which goes to $0$, we obtain the eigencurve $\Cc^r \rightarrow \W$. When $r=0$ this is the Coleman-Mazur eigencurve which we shall denote by $\Cc$.\\
For a Banach module $M$, a completely continuous operator $U$ and $\alpha \in \Q_{\geq 0}$, we define ${M}^{\leq \alpha}$ resp. ${M}^{> \alpha}$ as the subspace which contains all the generalized eigenspaces of eigenvalues of $U$ of valuation less or equal than $\alpha$ resp. strictly bigger than $\alpha$. Then the above discussion gives us the following proposition which is essentially all we need in what follows;
\begin{prop}\label{lessalpha}
 For all $\alpha \in \Q_{> 0}$ we have the a direct sum decomposition
\begin{align*}
 {\N^{r}(N,\A(\U))}={\N^{r}(N,\A(\U))}^{\leq \alpha} \bigoplus {\N^{r}(N,\A(\U))}^{> \alpha},
\end{align*}
where ${\N^{r}(N,\A(\U))}^{\leq \alpha}$ is a finite dimensional, free Banach module over $\A(\U)$. Moreover the projector to ${\N^{r}(N,\A(\U))}^{\leq \alpha}$ is given by a formal series in $U_p$ which we shall denote by $\mathrm{Pr}^{\leq \alpha}$.
\end{prop}
\begin{rem}\label{remv}
 As ${\N^{r}(N,\A(\U))}^{\leq \alpha}$ is of finite rank and $\A(\U)$ is noetherian there exists $v$ such that ${\N^{r}(N,\A(\U))}^{\leq \alpha} ={H^{0}(X(v)\times \U,\h_{\kappa}^{r}\hat{\otimes} \oo_{\U})}^{\leq \alpha} $.
\end{rem}
If we want to consider forms with  Nebentypus $\psi$ whose $p$-part is non-trivial, we need to apply the above construction to an affinoid $\U$ of $\W$ where $\psi$ is constant. This is because finite-order characters do not define Tate functions on $\W$.\\
It is well known that on a finite dimensional vector space over a complete field, all the norms are equivalent. In particular the overconvergent norm on ${\N^{r}(N,\A(\U))}^{\leq \alpha}$ is equivalent to sup-norm on the coefficients of the $q$-expansion. We call it the $q$-expansion norm; a unit ball for this norm defines a natural integral structure ${\N^{r}(N,\A(\U))}^{\leq \alpha}$, which coincides with the one defined in the previous subsection.\\ 
We now give  a useful lemma.
\begin{lemma}\label{qexpandproj}
 Let $f$ be a nearly overconvergent form in  $\N_k^{r,\dagger}(N)$ and let $f^{\leq \alpha}$ be its projection to ${\N_k^{r,\dagger}(N)}^{\leq \alpha}$. If $f(q,X) \in p^n\Z_p[[q]][X]$, then $f^{\leq \alpha}(q,X) \in p^n\Z_p[[q]][X]$.
\end{lemma}
\begin{proof}
 Let $f$ be as in the statement of the lemma, then we have $U_p f(q,X)\in p^n\Z_p[[q]][X]$. As $f^{\leq \alpha} = \mathrm{Pr}^{\leq \alpha}f$, we conclude. 
\end{proof}
Let $\A^{0}(\U)$ be the unit ball in $\A(\U)$. We have the following proposition which, roughly speaking, guarantees us that the limit for the $q$-expansion norm of nearly overconvergent forms of bounded slope is nearly overcovergent.
\begin{prop}\label{naivefam}
 Let $F(\kappa) = \sum_{i=0}^r F_i(\kappa) X^i$, with $F_i(\kappa)$ in $\A^{0}(\U)[[q]]$. Suppose that for a set $\left\{ \kappa_i \right\} $ of $\Qb_p$-points of $\U$ which are dense we have $F(\kappa_i) \in {\N_{\kappa_i}^{r}(N,\Qb_p)}^{\leq \alpha}$. Then 
\begin{align*}
 F(\kappa) \in {\N^{r}(N,\A(\U))}^{\leq \alpha}.
\end{align*}
\end{prop}
\begin{proof}
It is enough to show that for every $\kappa_0$ in $\U$, $F(\kappa_0)$ is nearly overconvergent (and the radius of overconvergence can be chosen independently of $\kappa_0$ by Remark \ref{remv}).\\ 
We follow the proof of  \cite[Corollary  4.8]{Til}.  We have for all $\kappa$ in $\U$ the Eisenstein series $E(\kappa)$. It is know that $E(\kappa)$ has no zeros on $X(v)$ for  $v > 0$ small enough. \\
We will write $f\equiv 0 \bmod p^n $ for $f(q,X) \in p^n\A^{0}(\U)[[q]][X]$. \\
Let $\kappa_0$ be a $L$-point of $\U$, and fix a sequence of points $\kappa_i$, $i > 0$, of $\U$ such that $\kappa_i$ converges to $\kappa_0$. In particular, $F(\kappa_i)$ converges to $F(\kappa_0)$ for the $q$-expansion topology. Let us consider the nearly overconvergent modular forms $G(\kappa_i):=\frac{F(\kappa_i)E(\kappa_0)}{E(\kappa_i)}$ of weight $\kappa_0$, we want to show that ${G(\kappa_i)}^{\leq \alpha}$ converge to $F(\kappa_0)$ in the $q$-expansion topology. This will prove that $F(\kappa_0)$ is nearly overconvergent because, as already said, in  the space of nearly overceonvergent forms of slope bounded by $\alpha$ all the norms are equivalent.\\
If $\vert \kappa_i - \kappa_0 \vert < p^{-n}$, we have  $E(\kappa_i)\equiv E(\kappa_0) \bmod p^n$, hence ${E(\kappa_i)}^{-1}\equiv {E(\kappa_0)}^{-1} \bmod p^n$; consequently, it is clear that  $G(\kappa_i) \equiv F(\kappa) \bmod p^n$. We apply Lemma  \ref{qexpandproj} to the forms $G(\kappa_i)-F(\kappa_i)$ to see that ${G(\kappa_i)}^{\leq \alpha}$ is  a sequence of overconvergent forms of weight $\kappa_0$ and bounded slope which converges to $F(\kappa_0)$ for the $q$-expansion topology.\\
\end{proof}

\begin{rem} In \cite[\S 1]{Pan} the author defines {\it rigid analytic nearly holomorphic modular forms} as  elements of $\A(\U)[[q]][X]$ which on classical points give classical nearly holomorphic forms. It would be interesting to compare his definition with the one here, especially understanding necessary and sufficient conditions to detect when a specialization at a non classical weight of a rigid analytic nearly holomorphic modular form is nearly overconvergent or not.
\end{rem}

One application of the above proposition is that it allows us to define a Maa\ss{}-Shimura operator of weight $\kappa$ as follows. Let us define
\begin{align*}
 \log(\kappa) = \frac{\log_p(\kappa(u^r))}{\log_p(u^r)}
\end{align*}
 for $u$ any topological generator of $1+p\Z_p$ and $r$ any integer big enough. \\
For any open affinoid $\U$ of $\W$ and $\kappa_0$ in $\W(\C_p)$, we define the $\kappa_0$-translate $\U\kappa_0$ of $\U$ as the composition $\U \rightarrow \W$ with $\W \stackrel{ \times \kappa_0 }{\rightarrow} \W$.\\
\begin{prop}\label{MSfam}
 We have an operator 
$$ \begin{array}{ccccc}
    \delta_{\kappa} : & {\N^{r}(N,\A(\U))}^{\leq \alpha} & \rightarrow & {\N^{r+1}(N,\A(\U[2]))}^{\leq \alpha+1} \\
 & \sum_{i=0}^r F_i(\kappa) X^i & \mapsto & \sum_{i=0}^r \Theta F_i(\kappa) X^i + (\log(\kappa) -i) F_i(\kappa)X^{i+1}
   \end{array}
$$
which is $\A(\U)$-linear. 
\end{prop}
Note that $\delta_{\kappa}$ is not $\A(X(v))$-linear.
\begin{proof}
 It is an application of the fact that for classical $\Qb_p$-points of $\U$ above $[k]$ we have $[k+2](\delta_{\kappa})=\delta_{k}[k]$ and Proposition \ref{naivefam}.
\end{proof}
We point out that there are other possible constructions of the Maa\ss{}-Shimura operator on nearly overconvergent forms which are defined on the whole space ${\N^{r}(N,\A(\U))}$ and not only on the part of finite slope. In \cite{HX}, the authors construct an overconvergent Gau\ss{}-Manin connection 
\begin{align*}
\h_{\kappa}^{r} \rightarrow \h_{\kappa[2]}^{r+1}
\end{align*}
  using the existence of the canonical splitting of ${\h^1_{\mathrm{dR}}}_{/X(v)}$ given by $E_2$ (which exists because $X(v)$ is affinoid, \cite[Appendix 1]{Katz}).\\
Let $r\geq 0$ be an integer, we define \begin{align*}
\log^{[r]}(\kappa)= \prod_{j=0}^{r-1}{\log(\kappa[-2r +j])}.
                                       \end{align*}

Let us denote by $\K(\U)$ the total fraction field of $\A(\U)$; we define 
\begin{align*}
{\N^{r}(N,\K(\U))}^{\leq \alpha} = & {\N^{r}(N,\A(\U))}^{\leq \alpha}\otimes_{\A(\U)}\K(\U).
\end{align*}

\begin{prop}\label{kappaMassS}
Let $F(\kappa) $ in ${\N_{\kappa}^{r}(N,\K(\U))}^{\leq \alpha}$, then \begin{align*}
F(\kappa)= \sum_{i=0}^r \frac{\delta_{\kappa[-2i]}^i G_i(\kappa)}{\log^{[i]}(\kappa)}
\end{align*}
for a unique sequence  $(G_i(\kappa))$, $i=0,\ldots,r$, with $G_i(\kappa)$ in $\M(N,\K(\U[-2i]))$.
\end{prop}
\begin{proof}
The proposition is clear if $r=0$. For $r \geq 1$, we proceed by induction; write 
\begin{align*}
F(\kappa)= & \sum_{i=0}^r F_i(\kappa) X^i,
\end{align*}  we have then $\eps^r F(\kappa) = r! F_r(\kappa)$, so $F_r(\kappa)$ is a family of overconvergent forms. \\
We pose $G_r(\kappa) := F_r(\kappa)$ and we see easily that \begin{align*}
F(\kappa) - \frac{\delta_{\kappa[-2r]}^r G_r(\kappa)}{\log^{[r]}(\kappa)}
\end{align*} has degree $r-1$ and by induction there exist $G_i(\kappa)$ as in the statement.\\ 
For uniqueness, suppose 
\begin{align*}
 \sum_{i=0}^r \frac{\delta_{\kappa[-2i]}^i G_i(\kappa)}{\log^{[i]}(\kappa)} = 0,
\end{align*}
by applying $\eps^r$ we obtain $G_r(\kappa)=0$ and uniqueness follows by induction.
\end{proof}
We have the following corollaries.
\begin{coro}
We have an isomorphism of Hecke-modules
\begin{align*}
 \bigoplus_{i=0}^r \delta_{\kappa[-2i]}^i {\M(N,\K(\U[-2i]))}^{\leq \alpha -i} \cong {\N^r(N,\K(\U))}^{\leq \alpha}.
\end{align*}
and consequently the characteristic series of $U_p$ is given by $Q_r(\kappa,T)= \prod_{i=0}^r Q_0(\kappa[-2i],p^iT)$. 
\end{coro}
\begin{coro}\label{holoproj}
We define  a projector 
\begin{align*}
H : {\N^r\left(N,\A(\U)\right)}^{\leq \alpha} \rightarrow {\M\left(N,\A(\U)\left[\frac{1}{\prod_{j=0}^{2r}{\log(\kappa[-j])}}\right]\right)}^{\leq \alpha} 
\end{align*}
 by sending $F(\kappa)$ to $G_0(\kappa)$. It is called the {\it overconvergent projection}. 
\end{coro}
\begin{proof}
We use the same notation of the proof of Proposition \ref{kappaMassS}. If $F(\kappa)(X)$ has $q$-expansion in $\A(\U)[[q]][X]$, we can see by induction on the degree that then the only possible poles of $G_0(\kappa)$ are the zero of $\prod_{j=0}^{2r}\log(\kappa[-j])$. 
\end{proof}
It is clear that this projector is a $p$-adic version of the classical holomorphic projector.\\
We remark that it is not possible to improve Proposition \ref{kappaMassS} allowing holomorphic coefficients, as shown by the following example; let us write $E_2^{\mathrm{cr}}(z)$ for the critical $p$-stabilization $E_2(z) - E_2(pz) $. We have that the polynomial $q$-expansion of  $E_2^{\mathrm{cr}}(z)$ is 
\begin{align*}
E_2^{\mathrm{cr}}(q,X)=\frac{p-1}{2p}X +\sum_{n p^m, (n,p)=1} p^{m}\sigma_1(n) q^{n p^m}.
\end{align*}
Recall the Eisenstein family $\tilde{E}(\kappa)$ defined in (\ref{GEis}). We have 
\begin{align*}
E_2^{\mathrm{cr}}(q,X) = & \delta_{\kappa}\tilde{E}(\kappa)|_{\kappa = \mathbf{1} }, 
\end{align*}
as the residue at $\kappa = \mathbf{1}$ of $\zeta^*(\kappa)$ is $\frac{p-1}{p}$.
The fact that the overconvergent projector has denominators in the weight variable was already known to Hida \cite[Lemma 5.1]{H1}.\\
We now give  the following proposition.
\begin{prop}\label{MSfamilies}
Let $F(\kappa)$ be an element of $\N^{r}(N,\A(\U))$ and suppose that $F(\kappa)$ is an eigenform for the whole Hecke algebra and of finite slope for $U_p$. Then $F(\kappa)=\delta_{\kappa}^r G(\kappa)$, for $G(\kappa) \in \M(N,(\K(\U[-2r])))$ a family of overconvergent eigenforms.
\end{prop}

\begin{proof}
Let $\lambda_F(n)$ be the Hecke eigenvalue of $T_n$; we have  from Proposition \ref{kappaMassS}  that 
\begin{align*}
F(\kappa)(X) = \sum_{i=0}^{r}\delta_{\kappa[-2i]}^i G_i(\kappa)
\end{align*}
with $G_i(\kappa)$ overconvergent. Moreover, we know from Proposition \ref{deltaT_l} that $G_i(\kappa)=a_0(G_i) \sum_{i=1}^{\infty} n^{-i} \lambda_F(n) q^n $. We have then 
\begin{align*}
G_i(\kappa)=\frac{a_0(G_i)}{a_0(G_r)}\Theta^{r-i} G_r(\kappa).
\end{align*}
By restriction to  the ordinary locus and projecting to $\omega^{\kappa}$ by $X \mapsto 0$ we find:
\begin{align*}
F(\kappa)(0)= & \left(\sum_{i=0}^r \frac{a_0(G_i)}{a_0(G_r)}\right) \Theta^r G_0(\kappa).
\end{align*}
This is the same $q$-expansion of ${\left(\sum_{i=0}^r \frac{a_0(G_i)}{a_0(G_r)}\right) }\delta_{\kappa[-2r]}^r G_r(\kappa)$; hence we can conclude by Proposition \ref{ordproj}.
\end{proof}
For any $\alpha < \infty$ and for $i=0,\ldots,r$ we define a map $s_i:\Cc^{\leq \alpha} \rightarrow {\Cc^i}^{\leq \alpha + i}$ induced by 
\begin{align*}
\delta_{\kappa}^i : {\M(N,\A(\U))}^{\leq \alpha}  \rightarrow  {\N^i(N,\A(\U[2i]))}^{\leq \alpha +i}.
\end{align*}
 The interest of the above proposition lies in the fact that it tells us that $\Cc^r $ minus a finite set of points (such as $E_2^{\mathrm{cr}}$) can be covered by the images of $s_i$. The images of these maps are not disjoint; it may indeed happen that two families of different degrees meet.\\ 
\begin{exam} Let  $k\geq 2$ be an integer, we have that $\delta_{1-k}^{k}=\Theta^k$ (see formula \ref{deltatheta}). It is well known that $\Theta^k$ preserve overconvergence \cite[Proposition 4.3]{ColCO}.  Let $F(\kappa)$ be a family of overconvergent forms of finite slope, then the specialization at $\kappa=[1-k]$ of the  nearly overconvergent family $\delta_{\kappa}^{k}F(\kappa)$ is overconvergent and consequently belongs to an overconvergent family.
\end{exam}
From the polynomial $q$-expansion principle for the degree of near holomorphicity \cite[Corollary 3.2.5]{UrbNholo}, we see that intersections between families of different degrees may happen only when the coefficients of the higher terms in $X$ of $\delta_{\kappa}^i$ vanish. It is clear from Formula \ref{deltatheta} that this can happen only for points $[1-k]$, for $i \geq k\geq 2$. Note that these points lie above the poles of the overconvergent projection $H$. This is not a coincidence; in fact $H(\delta_{\kappa}^{k}F(\kappa))=0$ for all $ \kappa \in \W$ for which $H$ is defined. If we could extend $H$ over the whole $\W$, we should have then $H\left(\delta_{[1-k]}^{k}F([1-k])\right)=0$ but we have just seen that  $\delta_{[1-k]}^{k}F([1-k])$ is already overconvergent.

\section{Half-integral weight modular forms and symmetric square $L$-function}
In this section we first recall the definition and some examples of half-integral weight modular forms. Then we use them to give an integral expression of $\Ll(s,\mathrm{Sym}^2(f),\xi)$. We conclude the section studying the Euler factor by which  $\Ll(s,\mathrm{Sym}^2(f),\xi)$ and  $L(s,\mathrm{Sym}^2(f),\xi)$ differ.
\subsection{Half-integral weight modular forms}\label{Halfint}
We recall the definition of half-integral weight modular forms. We define a holomorphic function on $\h$
\begin{align*}
 \theta(z)=\sum_{n \in \Z} q^{n^2}, \;\;\; q=e^{2 \pi i z}.
\end{align*}
Note that this theta series has no relations with the operator $\Theta$ of the previous section. We hope that this will cause no confusion.\\
We define a factor of automorphy
\begin{align*}
 h(\gamma,z)= \frac{\theta(\gamma(z))}{\theta(z)},\;\; \gamma \in \G_0(4), z \in \h.
\end{align*}
It satisfies
\begin{align*}
 {h(\gamma,z)}^2 = \sigma_{-1}(d)(c+zd).
\end{align*}
Let $k\geq 0$ be an integer and $\G$ a congruence subgroup of $\Sl(\Z)$. We define the space of half-integral weight nearly holomorphic modular forms $\N^r_{k+\frac{1}{2}}(\G,\C)$ as the set of $\Cc^{\infty}$-functions
\begin{align*}
 f : \h \rightarrow \C
\end{align*}
such that \begin{itemize}
          \item $f|_{k+\frac{1}{2}}\gamma(z):= f(\gamma(z)){h(\gamma,z)}^{-1} {(c+zd)}^{-k}= f(z)$ for all $\gamma$ in $\G$,
          \item $f$ has a finite limit at all cusps of $\G$,
          \item there exist holomorphic  $f_i(z)$ such that 
  \begin{align*}
   f(z) = \sum_{i=0}^r f_i(z)\frac{1}{{(4 \pi y)}^i}, \; \; \; y = \mathrm{Im}(z).
  \end{align*}

          \end{itemize}
When $r=0$, one simply writes $\M_{k+\frac{1}{2}}(\G,\C)$  for the space of holomorphic forms of weight $k + \frac{1}{2}$.\\
As $\G$ is a congruence subgroup, then there exists $N$ such that each $f_i(z)$ as above admits a Fourier expansion of the form
\begin{align*}
 f_i(z) = \sum_{n=0}^{\infty} a_n(f_i)q^{\frac{n}{N}}.
\end{align*}
This allows us to embed $\N_{k+\frac{1}{2}}^r(\G,\C)$ into $\C[[q]][q^{\frac{1}{N}},X]$. 
For all $\C$-algebra $A$ containing the $N$-th roots of unity, we define 
\begin{align*}
\N_{k+\frac{1}{2}}^r(\G,A) = & \N^r_{k+\frac{1}{2}}(\G,\C) \cap A[[q]][q^{\frac{1}{N}},X].
\end{align*}
For a geometric definition, see \cite[Proposition 8.7]{DT}.
In the following, we will drop the variable  $z$ from $f$.\\  
Let us consider a non trivial Dirichlet character $\xi$ of level $N$ and let $\beta$ be $0$ resp. $1$ if $\xi$ is even, resp. odd. We define 
\begin{align*}
 \theta(\xi)=\sum_{n=1}^{\infty} n^{\beta}\xi(n)q^{n^2} \in \M_{\beta+\frac{1}{2}}(\G_1(4N^2),\xi,\Z[\zeta_N]).
\end{align*}
Another example of half-integral weight forms is given by Eisenstein series; we recall their definition. Let $k>0$ be an  integer and $\chi$ be a Dirichlet character modulo $Dp^r$ ($D$ a positive integer prime to $p$) such that $\chi(-1)={(-1)}^{k-1} $, we set
\begin{align*}
\frac{E^*_{k-1/2}(z,s;\chi)}{L(2s+2k-2,\chi^2)} & = \sum_{\huge{\gamma \in \G_{\infty}\setminus \G_0(Lp^r)}} \chi \sigma_{Dp^r}{\sigma_{-1}}^{k-1}(\gamma)
{h(\gamma,z)}^{-2k+1}{|h(\gamma,z)|}^{-2s},\\
E_{k-1/2}(z,m;\chi) & = C_{m,k} \left \{{(2 y)}^{\frac{-m}{2}} E^*_{k-1/2}(z,-m;\chi) \right \}|_{k-1/2} \tau_{Dp^r},\\
C_{m,k} & = {(2\pi)}^{\frac{m-2k+1}{2}}{(Dp^r)}^{\frac{2k-1-2m}{4}}\G\left(\frac{2k-1-m}{2}\right),
\end{align*}
where $\tau_{Dp^{r}}$ is the Atkin-Lehner involution for half-integral weight modular forms normalized as in \cite[\S2 h4]{H6} and $\sigma_n$ is the quadratic character corresponding via class field theory to the quadratic extension $\Q(\sqrt{n})/\Q$. \\
If we set $E_{k-1/2}(\chi)=E_{k-1/2}(z,3-2k;\chi)$, then $E_{k-1/2}(\chi)$ is a holomorphic modular form of half-integral weight $k-1/2$, level $Dp^r$ and nebentypus $\chi$. Let us denote by $\mu$ the M\"{o}bius function. The Fourier expansion of $E_{k-1/2}(\chi)$ is given by
\begin{align*}
L_{Dp}(3-2k,\chi^2) + \sum_{n=1}^{\infty} q^n L_{Dp}\left(2-k,\chi\chi_n\right) 
\sum_{\tiny{\begin{array}{c} t_1^2t_2^2 |n, \\ (t_1t_2,Dp)=1, \\ t_1>0, t_2>0 \end{array}}} \mu(t_1)\chi(t_1t_2^2)\chi_n(t_1)t_2{(t_1t_2^2)}^{k-2},
\end{align*}
where $L_{Dp}(s,\chi) = \prod_{q\mid Dp}(1-\chi_0(q)q^{-s})L(s,\chi_0) $, for $\chi_0$  the primitive character associated to $\chi$ .\\
Let $s$ be an odd integer, $1 \leq s \leq k-1$. We have the following key formula, for the compatibility with the Maa\ss{}-Shimura operators as defined in Section \ref{analytic};
\begin{align*}
 \delta_{k-s+\frac{1}{2}}^{\frac{s+1}{2}-1} E_{k-s+\frac{1}{2}}(\chi) = E_{k -\frac{1}{2}}(z,2k-s-2;\chi). 
\end{align*}
In particular $E_{k -\frac{1}{2}}(z,2k-s-2; \chi) \in \N^{\frac{s+1}{2}-1}_{k-\frac{1}{2}}(\G_1(Dp^r),\chi,\Qb)$.\\
If $g_1$ resp. $g_2$ denotes a form in $\N^{r_1}_{k_1+\frac{1}{2}}(\G_1(N),\psi_1,A)$ resp. $\N^{r_2}_{k_2-\frac{1}{2}}(\G_1(N),\psi_2,A)$,  then $g_1g_2$ belongs to  $\N^{r_1+r_2}_{k_1+k_2}(\G_1(N),\psi_1\psi_2\sigma_{-1},A)$.
\subsection{An integral formula}
In this subsection  we use the half-integral weight forms we have defined before to express $\Ll(s,\mathrm{Sym}^2(f),\xi)$ as the Petersson product of $f$ with the product of two half-integral weight forms.
Let $f$ be a cusp form of integral weight $k$ and Nebentypus $\psi_1$, $g$ a modular form of half-integral weight $l/2$ and Nebentypus $\psi_2$.  Let $N$ be the least common multiple of the levels of $f$ and $g$ and suppose $k > l/2$. We define the Rankin product of $f$ and $g$ 
$$D(s,f,g)=L_N(2s-2k-l+3,{(\psi\xi)}^2)\sum_n\frac{a(n,f)a(n,g)}{n^{s/2}}.$$ 
The Eisenstein series introduced above  allow us to give an integral formulation for this Rankin product \cite[Lemma 4.5]{H6}.
\begin{lemma}\label{RankPet}
 Let $f$, $g$ and $D(s,f,g)$ as above. Let $f^c=\overline{f(-\overline{z})}$. We have the equality
\begin{align*}
&{(4\pi)}^{-s/2}\G(s/2)D(s,f,g) =  \lla f^c,gE^*_{k-l/2}(z,s+2-2k;\psi_1\psi_2\sigma_{-N})y^{(s/2)+1-k} \rra_N, \\
&= {(-i)}^k \lla f^c|_k\tau_{N}, g|_{l/2}\tau_{N}\left(E^*_{k-l/2}(z,s+2-2k;\psi_1\psi_2\sigma_{-N})y^{(s/2)+1-k}\right)|_{k-l/2}\tau_N \rra_N.  
\end{align*}
\end{lemma}
Here $\lla f, g \rra$ denotes the complex Petersson product 
\begin{align*}
 \lla f, g \rra = \int_{X(\G)} \overline{f(z)} g(z) y^{k-2} \textup{d}x\textup{d}y
\end{align*}
and it is defined for any couple $(f,g)$ in ${\N_k^r(N,\C)}^2$ such that at least one between $f$ and $g$ is cuspidal.\\
If we take for $g$ a theta series $\theta(\xi)$ as defined above we have then 
\begin{align*}
D(\beta+s,f,\theta(\xi))= \Ll(s,\mathrm{Sym}^2(f),\xi),
\end{align*}
for $\Ll(s,\mathrm{Sym}^2(f),\xi)$ the imprimitive $L$-function defined in the introduction. The interest of writing $\Ll(s,\mathrm{Sym}^2(f),\xi)$ as a Petersson product lies in the fact that such a product, properly normalized, is algebraic. This allowed Sturm to show Deligne's conjecture for the symmetric square \cite{Stu} and it is at the base of our construction of $p$-adic $L$-function for the symmetric square. We conclude with the following relation which can be easily deduced from \cite[(5.1)]{H6} and which is fundamental for the proof of Theorem \ref{MainThOC}.
\begin{lemma}\label{thetanonprim}
Let $f$ be an Hecke eigenform of level divisible by $p$ and let $\xi$ be a character defined modulo $Cp$ of conductor $C$. Let us denote by $\xi'$ the primitive character associated to $\xi$, then 
\begin{align*}
 D(s,f,\theta(\xi))= (1 - \lambda_p^2 p^{1-s}) D(s,f,\theta(\xi')).
\end{align*}
\end{lemma}

\subsection{The $L$-function for the symmetric square}\label{primLfun}
Let $f$ be a modular form of weight $k$  and of Nebentypus $\psi$ and   $\pi(f)$ the automorphic representation of $\gl(\Aa)$ spanned by $f$. Let us denoted by $\lambda_q$ the associated set of Hecke eigenvalues. In \cite{GJ}, the authors construct an automorphic representation of ${\mbox{GL}_3}(\Aa)$ denoted  $\hat{\pi}(f)$ and usually called the base change to ${\mbox{GL}_3}$ of $\pi({f})$. It is standard to associate to $\hat{\pi}({f})$ a complex $L$-function $\LL(s,\hat{\pi}(f))$ which satisfies a nice functional equation and coincides with $\Ll(s,\mathrm{Sym}^2(f),\psi^{-1})$ up to some Euler factors. The problem is that some of these Euler factors could vanish at critical integers. We recall very briefly the $L$-factors at primes of bad reduction of $\LL(s,\hat{\pi}(f))$ in order to determine in Section \ref{Benconj} whether the $L$-value we interpolate vanishes or not. We shall also use them in the Appendix \ref{FE} to generalize the results of \cite{DD,H6}. For a more detailed exposition, we refer to 
\cite[\S 4.2]{RosH}.
Fix an adelic Hecke character $\tilde{\xi}$ of $\Aa_{\Q}$ and denote by $\xi$ the corresponding (primitive) Dirichlet character. For any place $v$ of $\Q$, we pose 
\begin{align*}
 L_v(s,\hat{\pi}(f),\xi) = & \frac{L_v(s,{\pi}{({f})}_v \otimes \tilde{\xi}_v \times \check{\pi}{({f})}_v) }{L_v(s, \tilde{\xi}_v) },
\end{align*}
 where $\check{\phantom{  }}$ denotes the contragredient and ${\pi}{(f)}_v  \times \check{\pi}{(f)}_v$ is a representation of $\gl(\Q_v)\times\gl(\Q_v)$. \\ 
The completed $L$-function \begin{align*}
 \LL(s,\hat{\pi}({f}),\xi) = & \prod_v L_v(s,\hat{\pi}(f),\xi) 
\end{align*}
is holomorphic over $\C$ except in a few cases which correspond to CM-forms with complex multiplication by $\xi$ \cite[Theorem 9.3]{GJ}.\\
Let $\pi = \pi (f)$ and let $q$ be a place where $\pi$ ramifies and let $\pi_{q}$ be the component at $q$. By twisting by a character of $\Q_{q}^{\times}$, we may assume that $\pi_{q}$ has  minimal conductor among its twists; this does not change the $L$-factor  $\hat{\pi}_q$. Let $\psi'$ be the Nebentypus of the minimal form associated with $f$.\\
We distinguish the following four cases:
\begin{itemize}
	\item[(i)]   $\pi_{q}$ is a principal series $\pi(\eta,\nu)$, with both $\eta$ and $\nu$ unramified,
	\item[(ii)]  $\pi_{q}$ is a principal series $\pi(\eta,\nu)$ with $\eta$ unramified and $\nu$ ramified,
	\item[(iii)] $\pi_{q}$ is a special representation $\sigma(\eta,\nu)$ with $\eta$, $\nu$ unramified and $\eta\nu^{-1} = |\phantom{e}|_{q}$,
	\item[(iv)]  $\pi_{q}$ is supercuspidal. 
\end{itemize}
We will partition the set of primes dividing the conductor of $f$ as $\Sigma_1,\cdots,\Sigma_4$ according to these cases.  When $\pi_{q}$ is a ramified principal series we have $\eta(q)=\lambda_q q^{\frac{1-k}{2}}$ and $\nu=\eta^{-1}\tilde{\psi'}_q$, where $\tilde{\psi'}$ is the adelic character corresponding to $\psi'$.
In case {\it i)}, if $ \tilde{\xi}_q $ is unramified, the Euler factor ${L_q(s,\hat{\pi}(f),\xi)}^{-1}$ is  
\begin{align*}
(1-\tilde{\xi}_q\nu^{-1}\eta(q){q}^{-s})(1-\tilde{\xi}_q(q){q}^{-s})(1-\tilde{\xi}_q\nu\eta^{-1}(q){q}^{-s})
\end{align*}
and $1$ otherwise.
In case {\it ii)} we have  that ${L_q(s,\hat{\pi}(f),\xi)}^{-1}$ equals
\begin{align*}
 (1-\tilde{\xi}_q {\tilde{\psi'}_q}^{-1}(q)\lambda^2_q q^{1 -k-s})(1-\tilde{\xi}_q(q){q}^{-s})(1-\tilde{\xi}_q\tilde{\psi'}_q(q)\lambda^{-2}_q q^{{k-1}-s})
\end{align*}
While in the third case if $\tilde{\xi}_{q}$ is unramified we have $(1-\tilde{\xi}_q(q){q}^{-s-1})$ and $1$ otherwise.
The supercuspidal factors are slightly more complicated and depend on the ramification of $\tilde{\xi}_{q}$. They are classified by \cite[Lemma 1.6]{Sc}; we recall them briefly. Let $q$ be a prime such that $\pi_{q}$ is supercuspidal. If $\tilde{\xi}_{q}^2$ is unramified, let $\lambda_1$ and $\lambda_2$ the two ramified characters such that $\tilde{\xi}_{q}\lambda_i$ is unramified. We consider the following disjoint subsets of $\Sigma_4$:
\begin{align*}
\Sigma_4^0 &=\left\{q \in \Sigma_4 : \tilde{\xi}_{q} \mbox{ is unramified and } \pi_{q}\cong\pi_{q}\otimes \tilde{\xi}_q  \right \},\\
\Sigma_4^1 &=\left\{q \in \Sigma_4 : \tilde{\xi}_{q}^2 \mbox{ is unramified and } \pi_{q}\cong\pi_{q}\otimes\lambda_i \mbox{ for } i=1,2 \right \},\\
\Sigma_4^2 &=\left\{q \in \Sigma_4 : \tilde{\xi}_{q}^2 \mbox{ is unramified and } \pi_{q}\not\cong\pi_{q}\otimes\lambda_1 \mbox{ and }\pi_{q}\cong\pi_{q}\otimes\lambda_2   \right \}, \\
\Sigma_4^3 &=\left\{q \in \Sigma_4 : \tilde{\xi}_{q}^2 \mbox{ is unramified and } \pi_{q}\not\cong\pi_{q}\otimes\lambda_2 \mbox{ and }\pi_{q}\cong\pi_{q}\otimes\lambda_1 \right \}.
\end{align*}
If $q$ is in $\Sigma_4$ but not in $\Sigma_4^i$, for $i=0,\cdots, 3$, then $L_q(s,\hat{\pi}(f),\xi)=1$.
If $q$ is in $\Sigma_4^0$, then $$ {L_q(s,\hat{\pi}(f),\xi)}^{-1}=1+\tilde{\xi}_q(q){q}^{-s}$$
and if $q$ is in $\Sigma_4^i$, for $i=1,2,3$ then 
$${L_q(s,\hat{\pi}(f),\xi)}^{-1}=\prod_{j \mbox{ s.t.} \pi_{q}\cong\pi_{q}\otimes\lambda_j} (1-\tilde{\xi}_{q}\lambda_j(q){q}^{-s}).$$
We try now to explain briefly these four cases. Suppose $q \neq 2$: then there exists a quadratic extension $F/\Q_q$ and a Galois character $\mu$ such that the local Galois representation $r_q(\pi_q)$ associated with $\pi_q$ is the induced from $F$ to $\Q_q$ of $\mu$. The explicit matrix for $\mathrm{Sym}^2(r_q(\pi_q))$ can be found in \cite[(4)]{HarCM} and involves $\mu^2$; these local $L$-factors are then compatible with the ones predicted by Deligne \cite[\S 1.1]{Del}.\\

If $v=\infty$, the $L$-factor depends only on the parity of the character by which we twist.  Let $\kappa=0,1$ according to the parity of $\xi_{\infty}\psi_{\infty}$, from \cite[Lemma 1.1]{Sc} we have  $L_{\infty}(s-k+1,\hat{\pi}(f),\xi\psi)=\G_{\mathbb{R}}(s-k+2 -\kappa)\G_{\mathbb{C}}(s)$ for the complex and real $\G$-functions
\begin{align*}
\G_{\mathbb{R}}(s) = & \pi^{-s/2}\G(s/2), \\
\G_{\mathbb{C}}(s) = & 2{(2\pi)}^{-s}\G(s).
\end{align*}
We define 
\begin{align*}\calE_{N}(s,f,\xi)=& \frac{ \prod_{q |N } (1-\xi(q)\lambda_q^2{q}^{-s}){L_{q}(s-k+1,\hat{\pi}(f),\xi\psi)}} {(1-\psi^2\xi^2(2)2^{2k-2-2s})}.
\end{align*}
Note that $\lambda_q=0$ if $\pi$ is not minimal at $q$ or if $\pi_{q}$ is a supercuspidal representation.
We multiply then  $\Ll (s,f,\xi)$, the imprimitive $L$-function, by $\calE_{N}(s,f,\xi)$ to get 
\begin{align*}
 L(s,\mathrm{Sym}^2(f),\xi):= & L(s-k+1,\hat{\pi}(f)\otimes \xi\psi) \\
 = & \Ll (s,f,\xi)\calE_{N}(s,f,\xi).
\end{align*}
We can now state the functional equation
\begin{align*}
\LL(s,\hat{\pi}(f),\xi) & =  \beps(s,\hat{\pi}(f),\xi)\LL(1-s,\hat{\pi}(f),\xi^{-1}),\\
\LL(s,\mathrm{Sym}^2(f),\xi) & =\beps(s-k+1,\hat{\pi}(f),\xi\psi)\LL(2k+1-s,\mathrm{Sym}^2(f^c),\xi^{-1}),
\end{align*}
for $\beps(s,\hat{\pi}(f),\xi)$ of \cite[Theorem 1.3.2]{DD}.

\section{$p$-adic measures and $p$-adic $L$-functions}\label{padicLfunc}
The aim of this section is to construct the $p$-adic $L$-functions which we have described in the introduction. We first review the notion of $h$-admissible distribution and we generalize  this notion to distributions with values in nearly overconvergent forms; we then produce two such distributions. We  shall  use these distributions to construct the $p$-adic $L$-functions for the symmetric square.
\subsection{Admissibility condition}
We now give the definition of the  admissibility condition for measures with value in the space of  nearly overconvergent modular forms. We will follow the approach of \cite[\S 3]{Pan}. Let us denote by $A$ a $\Q_p$-Banach algebra, by $M$ a Banach module over $A$ and by $Z_D$ the $p$-adic space ${(\Z/ Dp\Z)}^{\times} \times (1 +p\Z_p)$. 
Let $h$ be an integer, we define $\Cc^h(Z_D,A)$ as the space of locally polynomial function on $Z_D$ of degree strictly less than $h$ in the variable $z_p \in 1+p\Z_p$. Let us define 
$\Cc^h_n(Z_D,A)$ as the space of  functions from $Z_D$ to $A$ which are polynomial of degree stricty less than $h$ when restricted to ball of radius $p^n$. It is a compact Banach space and we have 
\begin{align*}
\Cc^h(Z_D,A) = \varinjlim_n \Cc^h_n(Z_D,A).
\end{align*}
If $h \leq h'$, we have an isometric immersion of $\Cc^h_n(Z_D,A)$ into $\Cc^{h'}_n(Z_D,A)$ \\
\begin{defin}
Let $\mu$ be an $M$-valued distribution on $Z_D$, i.e. a $A$-linear continuous map 
\begin{align*}
\mu : \Cc^1(Z_D,A) \rightarrow M.
\end{align*}
We say that $\mu$ is an $h$-admissible measure if $\mu$ can be extended to a continuous morphism (which we shall denote by the same letter) $\mu : \Cc^h(Z_D,A) \rightarrow M$ such that for all $n$ positive integer, any $a \in {(\Z/ Dp^n\Z)}^{\times}$ and $h'=0,\ldots, h-1 $ we have 
\begin{align*}
\left|\int_{a+(Dp^n)}{(z_p-a)}^{h'} \textup{d}\mu \right| = o(p^{-n(h'-h)}).
\end{align*} 
\end{defin}
Let us denote by $\mathbf{1}_U$ the characteristic function of a open set $U$ of $Z_D$, we shall sometimes write $\int_{U}  \textup{d}\mu$ for $\mu(\mathbf{1}_U)$.\\  
The definition of $h$-admissible measure for $A=\oo_{\C_p}$ is due to Amice and V\'elu.\\
There are many different (equivalent) definitions of a $h$-admissible measure; we refer to \cite[\S II ]{ColAnp} for a detailed exposition of them.\\
The following proposition will be very usefull in the following (\cite[Proposition II.3.3]{ColAnp});
\begin{prop}\label{hadm'}
 Let $\mu$ be a $h$-admissible measure, let  $\tilde{h} \geq h$ be a positive integer,  then $\mu$ satisfies 
\begin{align*}
 \left\vert \int_{a + Dp^n} {(z_p -a_p)}^{h'} \textup{d}\mu \right\vert =o(p^{-n(h'-\tilde{h})})
\end{align*}
for any $n \in \mathbb{N}$, $h' \in \mathbb{N}$ and  $a \in   {(\Z/ Dp^n\Z)}^{\times}$.
\end{prop}

It is known that any $h$-admissible measure is uniquely determined by the values $\int_{Z_D}\chi(z)\eps(z_p){z_p}^{h'} \textup{d}\mu$, for all integers $h'$ in $[0,\ldots,h-1]$, all $\chi$ in $\widehat{{(\Z/ Dp\Z)}^{\times}} $ and all finite-order characters $\eps$ of $1+p\Z_p$.\\
Let us fix now $\U$, an affinoid subset of  $\W$. The following proposition about the behavior of $U_p$ on $\N^r(Np^n,\A(\U))$  can be proven exactly as  \cite[Proposition 1.6]{Pan}, 
\begin{prop}
 Let $n\geq 1$ be an integer, we have that $U_p^n$ sends $\N^r(Np^{n+1},\A(\U))$ into $\N^r(Np,\A(\U))$. In particular, the map 
$$
\begin{array}{cccc}
\mathrm{Pr}^{\leq \alpha, p^{\infty}} : & \bigcup_{n=0}^{\infty}  \N^r(Np^{n+1},\A(\U)) & \rightarrow & \bigcup_{n=0}^{\infty} \N^r(Np^{n+1},\A(\U)) \\
 & G(\kappa) & \mapsto & U_p^{-n} \mathrm{Pr}^{\leq \alpha} U_p^n G(\kappa) . 
\end{array}
$$
is well-defined and induces an equality ${\N^r(Np^{n+1},\A(\U))}^{\leq \alpha} ={\N^r(Np,\A(\U))}^{\leq \alpha}$ 
\end{prop}
We remark that the trick to use $U_p$ to lower the level from $\G_0(p^n)$ to $\G_0(p^{n-1})$ was already known to Shimura and is a fundamental tool in the study of family of $p$-adic modular forms.  
We conclude the section with the following theorem, which is exactly \cite[Theorem 3.4]{Pan} in the nearly overconvergent context. Let $\U$ be an open affinoid of $\W$; we let $A=\A(\U)$ and $M=\N^r(\G,\A(\U))$.
\begin{theo}\label{Theoadm}
Let $\alpha$ be a positive rational number and let $\mu_{s}$, $s=0,1,\ldots$, be a set of distributions on $\Cc^1(Z_D,M)$. Suppose there exists a positive integer $h_1$ such that the following two conditions are satisfied:
\begin{align*}
\mu_{s}(a+(Dp^n)) \in {\N^{r}\left(Dp^{h_1n},\A(\U)\right)}, \\
\left\vert U^{h_1 n}\sum_{i=0}^s{ s \choose i } (-a_p)^{s-i}\mu_i(a+(Dp^n))\right\vert_p < C p^{-ns}.
\end{align*}
 Let $h$ be such that $h > h_1 \alpha +1$; then there exists a $h$-admissible measure $\mu$ such that 
\begin{align*}
\int_{a+(Dp^n)}(z_p-a_p)^{s}\textup{d}\mu = U_p^{-h_1 n} \mathrm{Pr}^{\leq \alpha}(U_p^{h_1 n}\mu_{s}(a+(Dp^n))).
\end{align*} 
\end{theo}

\subsection{Nearly overconvergent measures}\label{nomeasures}
In this subsection we will define two measures with values in the space of nearly overconvergent forms.
We begin by  studying the behavior of the Maa\ss{}-Shimura operator modulo $p^n$. We have from \cite[(6.6)]{H1bis} the following expression;
\begin{align}\label{deltatheta}
\delta_k^s = \sum_{j=0}^s { s \choose j} \frac{\G(k+s)}{\G(k+s-j)} {\Theta}^{s-j} X^j,
\end{align}  
for $\Theta = q\frac{\textup{d}}{\textup{d}q}$ as in Section \ref{overconvergen}.
We now give  an elementary lemma
\begin{lemma}
We have for all  integers $s$
\begin{align*}
\mathrm{v}_p(s) & \leq \mathrm{v}_p \left( {s \choose j} (k+s-1)\cdots (k+s -j)\right)
\end{align*}
for all $1 \leq j \leq s$.
\end{lemma}
\begin{proof}
Simply notice that the valuation of $(k+s-1)\cdots (k+s -j)$ is bigger than that of $j!$ and $s | (s! / (s-j)!)$.
\end{proof}
The following two propositions are almost straightforward;
\begin{prop}\label{MSmodp_0}
Let $k$, $k'$ be two integers, $k \equiv k' \bmod p^n(p-1)$ and $f_k$ and $f_{k'}$ two nearly holomorphic modular forms, algebraic such that 
$f_k \equiv f_{k'} \bmod p^m$. Then $\delta_k f_k \equiv \delta_{k'}f_{k'} \bmod p^{\mathrm{min}(n,m)}$. 
\end{prop}
\begin{proof}
Direct computation from the formula in Proposition \ref{MSfam}.
\end{proof}
\begin{prop}\label{MSmodp}
Let $k$, $k'$ be two integers, $k \equiv k' \bmod p^n(p-1)$ and $f_k$ and $f_{k'}$ two nearly holomorphic modular forms, algebraic of same degree such that 
$f_k \equiv f_{k'} \bmod p^n$. Let $s$, $s'$ be two positive integers, $s'=s+s_0 p^n(p-1)$. Then $(\delta_k^s f_k)|\iota_p \equiv  \delta_{k'}^{s'} f_{k'} \bmod p^{n}$.
\end{prop}
\begin{proof}
Iterating the above proposition we get  $\delta_k^s f_k \equiv  \delta_{k'}^{s} f_{k'} \bmod p^{n}$. 
But $s'-s \equiv 0 \bmod p^n$, so by the above lemma and (\ref{deltatheta}) we have $\delta_{k+2s}^{s'-s}\delta_k^s f_k\equiv {\Theta}^{s'-s}\delta_{k'}^{s}  f_{k'} $. We conclude as ${\Theta}^{s'-s}\equiv \iota_p \bmod p^n$.
\end{proof}

Before constructing the aforementioned measures, we recall the existence of the Kubota-Leopoldt $p$-adic $L$-function.
\begin{prop}\label{zetames}
Let $\chi$ be a primitive character modulo $Cp^r$, with $C$ and $p$ coprime and $r\geq 0$. Then for any $b\geq 2$ coprime with $p$, there exists a measure $\zeta_{\chi,b}$ such that for every finite-order character $\eps$ of $Z_D$ and any integer $m\geq 1$ we have
$$\int_{Z_D} \eps(z)z_p^{m-1}\textup{d}\zeta_{\chi,b}(z)=(1-\eps'\chi'(b)b^{m})L_{Dp}(1-m,\chi\eps), $$
where $\chi'$ denote the prime-to-$p$ part of $\chi$. 
\end{prop}
To such a measure and to each character $\eps$ modulo $Np^r$, we can associate by $p$-adic Mellin transform a formal series  
\begin{align*}
G(S,\eps,\chi,b) = \int_{Z_D} \eps(z) {(1+S)}^{z_p} \textup{d}\zeta_{\chi,b}(z)
\end{align*}
 in $\oo_K[[S]]$, where $K$ is a finite extension of $\Q_p$. We have a natural map from $\oo_K[[S]]$ to $\A(\W)$ induced by $S\mapsto (\kappa \mapsto \kappa(u)-1)$. We shall denote by $L_p(\kappa,\eps,\chi,b)$ the image of $G(S,\eps,\chi,b)$ by this map.\\
We define an element of $\A(\W)[[q]]$
\begin{align*}
\calE_{\kappa}(\eps)  & = \sum_{n=1, (n,p)=1}^{\infty} L_p(\kappa[-2],\eps, \sigma_n,b) 
q^n \sum_{\tiny{\begin{array}{c} t_1^2t_2^2 |n, \\ (t_1t_2,Dp)=1, \\ t_1>0, t_2>0 \end{array}}}t_1^{-2}t_2^{-3} \mu(t_1)\eps(t_1t_2^2)\sigma_n(t_1) \kappa({t_1t_2^2}).
\end{align*}
If $\kappa=[k]$, we have then $[k](\calE_{\kappa}(\eps))=(1-\eps'(b)b^{k-1})E_{k -\frac{1}{2}}(\eps\omega^{-k})|\iota_p$, where $\iota_p$ is the trivial character modulo $p$.\\
We fix two even Dirichlet characters: $\xi$ is primitive modulo ${\Z/Cp^{\delta}\Z}$ ($\delta=0,1$) and $\psi$ is defined modulo ${\Z/pN\Z}$. Fix also a positive slope $\alpha$ and an integer $D$ which is a square and divisible by $4$, $C^2$ and $N$.\\
Let $h$ be an integer, $h > 2{\alpha}+1$. For $s=0, 1,\ldots$ we now define  distributions $\mu_{s}$ on $\Z^{\times}_p$ with value in  ${\N^{r}(D,\A(\W))}^{\leq \alpha}$. For any finite-order character $\eps$ of conductor $p^n$ we pose 
\begin{align*}
 \mu_{s}(\eps) = \mathrm{Pr}^{\leq{\alpha}}U_p^{2n-1}\left(\theta(\eps\xi\omega^{s})|\left[\frac{D}{4C^2}\right]\delta_{\kappa\left[-s-\frac{1}{2}\right] }^{\frac{s-\beta_s}{2}} \calE_{\kappa[-s]}(\psi\xi \eps \sigma_{-1})\right)
\end{align*}
 with $\beta_s =0$, $1$ such that $s\equiv \beta_s \bmod 2$. The projector $\mathrm{Pr}^{\leq{\alpha}}$ is {\it a priori} defined only on $\N^{r}(D,\A(\U))$ but, thanks to Proposition \ref{lessalpha}, it makes perfect sense to apply it to a formal polynomial $q$-expansion as it is a formal power series in $U_p$ and we know how $U_p$ acts on a polynomial $q$-expansion.\\ 
Define $t_0 \in \Q$ to be the smaller rational such that $z_p^{\log(\kappa)}$ converges for all $z_p$ in $1+p\Z_p$ and $\kappa$ in $\W(t_0)$.
\begin{prop}\label{GlueDist}
 The distributions $\mu_s$ defined above define an $h$-admissible measure $\mu$ with values in ${\N^{r}(D,\A(\W(t_0) \times \W))}^{\leq \alpha}$.
\end{prop}
\begin{proof}
 We have to check that the two conditions of Theorem \ref{Theoadm} are verified. The calculations are similar to the one of \cite[Theorem 2.7.6, 2.7.7]{DD} or, more  precisely, to the one made by \cite[\S 3.5.6]{Gorsse} and \cite[\S 4.6.8]{CP} which study in detail the growth condition.
We have the discrete Fourier expansion
\begin{align*}
 \mathbf{1}_{a_p + p^n\Z_p} (x) = \frac{1}{p^{n-1}}\sum_{\eps} \eps(a_p^{-1}x).
\end{align*}
By integration, together with the fact that each  $\mu_{s}(\eps)$ belongs to  $\N^{r}(Dp^{2r},\A(\U))$, we obtain {\it i)}.\\
For the estimate {\it ii)}, we have to show that for all $n\geq 0$, $0 \leq  s \leq h-1 $
\begin{align*}
\left| U_p^{2n}\sum_{i=0}^s{ s \choose i } (-a_p)^{s-i}\mu_i(a+(Lp^n))\right|_p < C p^{-ns}
\end{align*}
where the norm $|\phantom{e}|_p$ is the $q$-expansion norm defined in Section \ref{families}. Let us write 
\begin{align*}
\sum_{i=0}^s{ s \choose i } (-a_p)^{s-i}\mu_i(a+(Lp^n)) = \sum_{j=0}^s \sum_{n=0}^{\infty} b_n^j(\kappa) X^j q^n.
\end{align*}  
Hence what we have to do is to bound the norm of $b_n^j=b_n^j(\kappa)$ on $\W(t_0)$. Using (\ref{deltatheta}), for $\beta_i=0,1$, $\beta_i \equiv i \bmod 2$,  we expand
\begin{align*}
& \mu_i(a+(Lp^n)) = \theta(a+(Lp^n)) \times \\
 &\times \sum_{j=0}^{\frac{i-\beta_i}{2}} { \frac{i-\beta_i}{2} \choose j} \log(\kappa[-\frac{i+1+\beta_i}{2}-1])\cdots  \log(\kappa[-\frac{i+1+\beta_i}{2}-j]) {\Theta}^{\frac{i-\beta_i}{2}-j}  \calE_{\kappa\left[-i\right]}(a+(Lp^n))X^j.
\end{align*}
Note that $\calE_{\kappa} = \sum_n \nu'_n q^n$, where $\nu'_n$ are measures. Hence we have  
\begin{align*}
b_n^j=\sum_{i=0}^s { s \choose i} {(-a_p)}^{s-i} a(i,n) { \frac{i-\beta_i}{2} \choose j} \log(\kappa[-\frac{i+1+\beta_i}{2}-1])\cdots  \log(\kappa[-\frac{i+1+\beta_i}{2}-j]),
\end{align*} 
where  $a(i,n)=\int z_p^i \textup{d}\nu_n $, for $\nu_n$ a measure, namely a linear combination of Kubota-Leopoldt $p$-adic $L$-functions and Dirac deltas. For $\beta=0$, $1$ and for $j\geq 1$  we shall write:
\begin{align*}
 D_{\kappa,\beta}^j = \left(z_p^{\log(\kappa^{2})-2 -\beta} \frac{\partial}{\partial z_p} \cdots z_p^{-1}\frac{\partial}{\partial z_p} \cdot  z_p^{-1}\frac{\partial}{\partial z_p}z_p^{1+\beta+2j -2\log(\kappa)} \right),
\end{align*} 
where we have applied $\frac{\partial}{\partial z_p}$ $j$-times and multiplied $j-1$ times by $z_p^{-1}$. We note that we have for any positive integer $i$:
\begin{align*}
D_{\kappa,\beta}^j (z_p^i) = \log(\kappa^{-2}[i+1+\beta+2])\cdots  \log(\kappa^{-2}[i+1+\beta+2j]) z_p^i.
\end{align*}
Similarly, 
\begin{align*}
 \Dfrak_{\kappa,\beta}^j & = \left(z_p^{2j+\beta -1} \frac{\partial}{\partial z_p} \cdots z_p^{-1}\frac{\partial}{\partial z_p} \cdot  z_p^{-1}\frac{\partial}{\partial z_p}z_p^{-\beta} \right),\\
 \Dfrak_{\kappa,\beta}^j (z_p^i) & = (i-\beta)(i-\beta -2)\cdots  (i-\beta-2j+2) z_p^i.
\end{align*} 
Summing up
\begin{align*}
 & \sum_{i=0}^s { s \choose i} {(-a_p)}^{s-i} { \frac{i-\beta}{2} \choose j} \log(\kappa[-\frac{i+1+\beta}{2}-1])\cdots  \log(\kappa[-\frac{i+1+\beta}{2}-j])z_p^i = \\
 & =  \frac{2^{2j}}{j!}{(-1)}^{-j}\Dfrak_{\kappa,\beta}^j D_{\kappa,\beta}^j \left({(z_p-a_p)}^s\right).
\end{align*}
We have that $|\frac{\partial}{\partial z_p}{(z_p-a_p)}^s\mathbf{1}_{a+(Lp^n)}|_p=p^{-n(s-1)}$. The $p$-adic logarithm  $\log(\kappa)$ is not bounded on $\W$ but the maximum modulus principle \cite[\S 3.8.1,  Proposition 7]{BGR} ensures us  that there exists $C_{t_0} \in \mathbb{R}$ such that $\vert \log(\kappa)\vert_{\W(t_0)} < C_{t_0}$. As $\nu_n$ is a measure, we have 
\begin{align*}
 \left\vert\int_{a+(Dp^n)} \Dfrak_{\kappa,\beta}^j D_{\kappa,\beta}^j \left({(z_p-a_p)}^s\right) \textup{d}\nu_n\right\vert_p = o(p^{-n(s-2j)}).
\end{align*}
As we have 
\begin{align*}
 \sum_{i=0, i\equiv 0 \bmod 2 }^s{ s \choose i } (-a_p)^{s-i}z_p^{i}= &\frac{1}{2}({(z_p-a_p)}^s + {(-z_p-a_p)}^s),\\
 \sum_{i=0, i\equiv 1\bmod 2 }^s { s \choose i } (-a_p)^{s-i}z_p^{i}= &\frac{1}{2}({(z_p-a_p)}^s - {(-z_p-a_p)}^s),
\end{align*}
we deduce the same estimate on $\vert b_n^j \vert_p$ because
\begin{align*}
\frac{j!}{2^{2j}}{(-1)}^{-j} b_n^j = & \int_{a+(Dp^n)} \Dfrak_{\kappa,0}^j D_{\kappa,0}^j \left(\frac{1}{2}({(z_p-a_p)}^s + {(-z_p-a_p)}^s)\right) \textup{d}\nu_n + \\
  & + \int_{a+(Dp^n)} \Dfrak_{\kappa,1}^j D_{\kappa,1}^j \left(\frac{1}{2}({(z_p-a_p)}^s - {(-z_p-a_p)}^s)\right) \textup{d}\nu_n.
\end{align*}

Recall that $U_p^{2n}X^j=p^{2nj}X^j$.   Then  for all $s\geq 0$, $n \geq 0$ we have the growth condition of Proposition \ref{Theoadm}. This assures us that these distributions define a unique $h$-admissible measure $\mu$ with values in $\A(\W(t_0) \times \W(t))[[q]]$. We can see using Proposition \ref{MSmodp} that  $\mu_{s}(\eps)$ satisfies the hypothesis of Proposition \ref{naivefam} and hence  $\mu_{s}(\eps)$ belongs to  ${\N^{r}(L,\A(\W(t_0)))}^{\leq \alpha}$. The Mellin transform 
\begin{align*}
 \kappa' \mapsto \int_{1+p\Z_p} {\kappa'(u)}^{z}\textup{d}\mu(z)
\end{align*}
gives us the desired two variables family.
\end{proof}
Note that if $\alpha=0$, we do not need to introduce the differential operators $\Dfrak_{\kappa,\beta}^j$ and $D_{\kappa,\beta}^j$ and the above families are defined over the whole $\W \times \W$ (see also the construction in \cite[\S 4.3]{UrbNholo}).\\
We define then an improved one variable family $\theta.E(\kappa)=\theta.E(b,\xi', \psi')$. We call this measure improved because it will allow us to construct a one-variable $p$-adic $L$-function which does not present a trivial zero. Fix a weight $k_0$ and, to define $\theta.E(\kappa)$, suppose that $\xi=\xi'\omega^{2-k_0}$, with $\xi'$ a character of conductor $C$ such that  $\xi'(-1)={(-1)}^{k_0}$. We define
\begin{align*}
  \mathrm{Pr}^{\leq \alpha}  {\left(\theta(\xi') |\left[ \frac{D}{4C^2} \right] \delta_{\kappa\left[- k_0 - \frac{3}{2}  \right]}^{\frac{k_0 -\beta}{2}-1}\tilde{\calE}_{[\kappa]}(\sigma_{-1}\psi'\xi')\right)},
\end{align*}
where 
\begin{align*}
\tilde{\calE}_{[\kappa]}(\chi') & =  (1-\chi'(b)\kappa(b)b^{-k_0+1}) L_p(\kappa^2[-4-2k_0],{\chi'}^2, \mathbf{1},b) + \\
 & (1-{(\chi')}^2(b)\kappa(b^2)b^{-2k_0 -4}) \sum_{n=1}^{\infty} L_p(\kappa[-k_0],\chi, \sigma_n,b) q^n \\ 
 & \times \sum_{\tiny{\begin{array}{c} t_1^2 t_2^2 |n, \\ (t_1 t_2,Dp)=1, \\ t_1>0, t_2>0 \end{array}}} t_1^{-2} t_2^{-3} \mu(t_1)\chi(t_1 t_2^2)\sigma_n(t_1) \kappa(t_1 t_2^2).
\end{align*}

Let $F(\kappa)$ be a family of overconvergent eigenforms with coefficients in $\U$. We define a linear form $l_{F}$ on  ${\M(N,\K(\U))}^{\leq \alpha}$  as in \cite[Proposition 6.7]{Pan}. Note that the evaluation formula holds also for weights which are, in Panchishkin's notation, {\it critical}, i.e. when $\alpha = (k_0-2)/2$, because at such point $\Cc$ is \'etale above $\W$. In \cite{Pan} this case is excluded because a trivial zero appears in his interpolation formula. Such a trivial zero is studied in \cite{SteTZ}, where Conjecture \ref{MainCoOC} for $\rho_f(k_0/2)$ is proven.\\
We can define linear forms for nearly overconvergent families, in a way similar to \cite[\S 4.2]{UrbNholo} but without the restriction $N=1$.
For this, let ${\T^r(N,\K(\U))}^{(Np)}$ be the sub-algebra of $\mathrm{End}_{\K(\U)}({\N^{r}(N,\K(\U))}^{\leq \alpha})$ generated by the Hecke operators  outside $Np$. It is a commutative and semisimple algebra; hence we can diagonalize ${\N^{r}(N,\K(\U))}^{\leq \alpha}$ for the action of this Hecke algebra. Let $F$ be an eigenform for ${\T^r(N,\K(\U))}^{(Np)}$, we have a linear form $l_{F}^r$ corresponding to the projection of an element of ${\N^{r}(N,\K(\U))}^{\leq \alpha}$ to the $\K(\U)$-line spanned by $F$. \\
We say that a family $F(\kappa)$ is {\it primitive} if it is a family of eigenforms and all its specializations at non critical weights are the Maa\ss{}-Shimura derivative of a primitive form. This implies that the system of eigenvalues for ${\T^r(N,\K(\U))}^{(Np)}$ corresponding to $F(\kappa)$ appears with multiplicity one in ${\N^{r}(N,\K(\U))}^{\leq \alpha}$. We can see $l_{F}^r$ as a $p$-adic analogue of the normalized Petersson product; more precisely we have  the following proposition. We recall that $\tau_N$  is the Atkin-Lehner involution of level $N$ normalized as in \cite[h4]{H6}. When the level will be clear from the context, we shall simply write $\tau$. \\
\begin{prop}
 Let $F(\kappa)$ be an overconvergent family of primitive eigenform of finite slope $\alpha$, degree $r$ and conductor $N$. Let  $k$ be a classical non critical weight, then for all $G(\kappa)$ in ${\N^{r}(N,\K(\U))}^{\leq \alpha}$ we have 
\begin{align*}
 l_{F}^r (G(k)) = \frac{\lla F(k)^{c}|\tau, G(k) \rra}{\lla F(k)^{c}|\tau, F(k) \rra}.
\end{align*}
\end{prop}
\begin{proof}
Let $f$ be an element of $\N^r_k(N,\C)$, the linear form 
\begin{align*}
 g \mapsto \frac{\lla f^{c}|\tau, g \rra}{\lla f^{c}|\tau, f \rra}
\end{align*}
is Hecke equivariant and takes the value $1$ on $f$. It is the unique one with these two properties.\\
For any pairs of forms $g_1$ and $g_2$ of weights $k-2r_1$ and $k-2r_2$, $r_1 \neq r_2$, using Proposition  \ref{sumMS} and Lemma \ref{deltaT_l}, we see that $\delta_{k-2r_1}^{r_1}g_1$ and $\delta_{k-2r_2}^{r_2}g_2$ are automatically orthogonal for the Petersson product normalized as above.\\ 
Then, as $k > 2r$, we have for any  $f$ in $\N_k^r(N,\C)$ a linear form 
\begin{align*}
 g \mapsto \frac{\lla f^{c}|\tau, g \rra}{\lla f^{c}|\tau, f \rra}
\end{align*}
which is Hecke equivariant and takes the value $1$ on $f$. Moreover if $f$ is defined over $\Qb$ then  both  linear forms are defined over $\Qb$.\\
Let  $l_{F(k)}^r$  be the specialization of $l_{F}^r$ at weight $k$. As we have $l_{F(k)}^r (F(k))=1$, we deduce that $l_{F(k)}^r$ must coincide, after extending scalars if necessary, with the previous one and  we are done.
\end{proof}
In particular, we deduce from the above proof the $p$-adic analogue of the theorem which say that holomorphic forms are orthogonal to Maa\ss{}-Shimura derivatives.\\
We have the following lemma
\begin{lemma}
Let $H$ be the overconvergent projector of Corollary \ref{holoproj} and $F(\kappa)$ a family of overconvergent primitive eigenforms, then 
\begin{align*}
l_{F} \circ H = l_{F}^r.
\end{align*}
\end{lemma}
\begin{proof}
Let us write $G(\kappa) =\sum_{i=0}^r \delta_{\kappa[-2i]}^i G_i(\kappa)$. The above proposition tells us $l^r_F(G(\kappa)) = l^r_F(G_0(\kappa))$. By definition, $l_{F}^r = l_{F}$ when restricted to ${\M(N,\K(\U))}^{\leq \alpha}$ and we are done.
\end{proof}
We remark that $l_{F}$ is defined over $\K(\U)$ but not over $\A(\U)$.\\ 
The linear forms $l_F$ defines a splitting of $\K(\U)$-algebras
\begin{align*}
 \T^r(N,\A(\W(t)))^{\leq \alpha}\otimes \K(\U) = \K(\U) \times  C
\end{align*}
and consequently an idempotent $1_F \in \T^r(N,\A(\W(t)))^{\leq \alpha}\otimes \K(\U) $. It is possible to find an element $H_F(\kappa) \in \A^{\circ}(\U)$ such that $H_F(\kappa) 1_F$ belongs to $ \T^r(N,\A(\W(t)))^{\leq \alpha}$. Then we can say that $l_F$ is not holomorphic in the sense that it is not defined for certain $\kappa$ in $\U$. We hope that the above lemma helps the reader to understand why the overconvergent projectors cannot be defined for all weights. \\
\begin{rem}
We will see in Section \ref{KimBel} some possible relations between the poles of $l_{F}$ and another $p$-adic $L$-function for the symmetric square.
\end{rem}

\subsection{The two $p$-adic $L$-functions}\label{padicL}
We shall now  construct the two-variable $p$-adic $L$-function $L_p(\kappa,\kappa')$ of Theorem \ref{Tintro} and, in the case where $\xi=\xi'\omega^{2-k_0}$, with $\xi'$ a character of conductor $C$ such that  $\xi'(-1)={(-1)}^{k_0}$, an {\it improved} $p$-adic $L$-function $L^{*}_p(\kappa)$. We call this $p$-adic $L$-function, in the terminology of Greenberg-Stevens, improved because it has no trivial zero and at $\kappa_0$ is a non zero multiple of the value $\Ll(k_0-1,\mathrm{Sym}^2(F(\kappa)),{\xi'}^{-1})$. \\
These two $p$-adic $L$-functions are related by the key Corollary \ref{CoroImp}. Allowing a cyclotomic variable forces us to use theta series of level divisible by $p$ even when the conductor of the character is not divisible by $p$; Lemma \ref{thetanonprim} tells us that the trivial zero for $f$ as in Theorem \ref{MainThOC} comes precisely from this fact.  The construction of the one-variable $p$-adic $L$-function is done in the spirit of \cite{HT}, using the measure $\theta.E(\kappa)$ which is not a convolution of two measures but a product of a measure by a constant theta series whose level is not divisible by $p$. We warn the reader that the proof of Theorem \ref{T1OC} below is very technical and is not necessary for the following.\\
Before constructing the $p$-adic $L$-functions, we introduce the generalization to nearly overconvergent forms of the {\it twisted} trace operator defined in \cite[\S 1 VI]{H1bis}.  It will allow us to simplify certain calculations we will perform later.\\
Fix two prime-to-$p$ integers $D$ and $N$, with $N|D$. We define for classical $k,r$
\begin{align*}
\begin{array}{ccccc}
 T_{D/N,k} : & {\N_k^{r}(Dp,A)} & \rightarrow & {\N_k^{r}(Dp,A)}\\
 & f & \mapsto &{(D/N)}^{k/2} \sum_{[\gamma] \in \G(N)/ \G(N,D/N)}  f |_k \left(\begin{array}{cc} 
1 & 0 \\
0 & D/N
\end{array} \right) |_k \gamma
\end{array}.
\end{align*}
As $D$ is prime to $p$, it is clear that $T_{D/N,k}$ commutes with $U_p$. It extends uniquely to a linear map
$$\begin{array}{ccccc}
 T_{D/N} : & {\N^{\infty}(Dp,\A(\U))} & \rightarrow & {\N^{\infty}(Np,\A(\U))}
\end{array}
$$
which in weight $k$ specializes to $T_{D/N,k}$. In particular, it preserves the slope decomposition. \\
Let us fix a $p$-stabilized eigenform $f$ of weight $k$ as in the introduction such that $k -1 > v_p(\lambda_p)$. Let $\Cc_F$ be a neighbourhood of $f$ in $\Cc$ contained in a unique irreducible component  of $\Cc$. It corresponds by duality to a family of overconvergent modular forms $F(\kappa)$. We have that the slope of $U_p$ on $\Cc_F$ is constant; let us denote by $\alpha$ this slope. We shall denote by $\lambda_p(\kappa)$ the eigenvalue of $U_p$ on $\Cc_F$.\\
 Let $u$ be a generator of $1+p\Z_p$ such that $u=b \omega^{-1}(b)$, where $b$ is the  positive integer we have chosen in Proposition \ref{zetames}.  Let us define \begin{align*}
\Delta(\kappa,\kappa') = & \left(1-\psi'\xi'(b)\frac{\kappa(u)}{b\kappa'(u)}\right), \\
\Delta_0(\kappa) = & (1-\xi'\psi'(b)b^{-k_0+1}\kappa(u))(1-\xi'\psi'(b)b^{-2k_0-4}{\kappa(u)}^2).
               \end{align*}
 The two $p$-adic $L$-functions that we define are 
\begin{align*}
L_p(\kappa,\kappa')= & D^{-1}{\Delta(\kappa,\kappa')}^{-1} l_{F}(T_{D/N} \theta \ast E(\kappa,\kappa')) \in \K(\Cc_F \times \W), \\
L_p^*(\kappa)= & D^{-1}{\Delta_0(\kappa)}^{-1} l_{F}( T_{D/N} \theta.E(\kappa)) \in \K(\Cc_F).
\end{align*}
We say that a point $(\kappa, \kappa')$ of $\A(\Cc_F \times \W)$ is classical if $\kappa$ is a non-critical weight and $\kappa'(z)=\eps(\lla z \rra) z^s$, for $\eps$ a finite-order character of $1+p\Z_p$  and $s$ an integer such that $ 1 \leq s+1 \leq k-1$. This ensures that $s+1$ is a critical integer \`a la Deligne for $\mathrm{Sym}^2(f)\otimes \omega^{-s}$.
We define certain numbers which will appear in the following interpolation formulae. Suppose that $(\kappa,\kappa')$ is classical in the above sense and let $n$ be such that $\eps$ factors through $1+p^n\Z_p$. Let $n_0=n$ resp. $n=0$ if $\eps$ is not trivial resp. is trivial. For a Dirichlet character $\eta$, we denote by $\eta_0$ the associated primitive character. Let us pose
\begin{align*} 
E_1(\kappa,\kappa')= & \lambda_p(\kappa)^{-2n_0}(1-{(\xi\eps\omega^{s})}_0(p)\lambda_p(\kappa)^{-2}p^{s});
\end{align*}
if $F(\kappa)$ is primitive at $p$ we define $E_2(\kappa,\kappa')=1$, otherwise
\begin{align*}
E_2(\kappa,\kappa')= & (1-{(\xi^{-1}\eps^{-1}\omega^{-s}\psi)}_0(p)p^{k-2-s}) \times \\
 & (1-{(\xi^{-1}\eps^{-1}\omega^{-s}\psi^{2})}_0(p)\lambda_p(\kappa)^{-2} p^{2k-3-s}).
\end{align*}
 We shall denote by $F^{\circ}(\kappa)$ the primitive form associated to $F(\kappa)$. We shall write $W'(F(\kappa))$ for the prime-to-$p$ part of the root number of $F^{\circ}(\kappa)$. If $F(\kappa)$ is not primitive at $p$ we pose 
\begin{align*}
S(F(\kappa)) = 
(-1)^k   \left( 1 - \frac{\psi_0(p)p^{k-1}}{\lambda_p(\kappa)^{2}} \right)\left( 1 - \frac{\psi_0(p)p^{k-2}}{\lambda_p(\kappa)^{2}} \right), 
\end{align*}
and $S(F(\kappa)) = (-1)^k  $ otherwise.
Let $D$ be a positive integer divisible by $4C^2$ and $N$, we shall write $D=4C^{2}D'$.
Let $\beta=0$, $1$ such that $s \equiv \beta \bmod 2$, we pose 
\begin{align*}
C_{\kappa,\kappa'} &  = s ! G(\xi\eps\omega^s) C(\xi\eps\omega^s)^{s} N^{-k/2} {D'}^{\frac{s-\beta}{2}}  2^{-2s -k -\frac{1}{2}},\\
C_{\kappa} & = C_{\kappa,[k_0-2]},
\end{align*}
where $C(\chi)$ denotes the conductor of $\chi$.
\begin{theo}\label{T1OC}
\begin{itemize}
\item[i)]
 The function $L_p(\kappa,\kappa')$ is defined  on $\Cc_{F} \times \W$, it is meromorphic in the first variable and of logarithmic growth $h=[2 \alpha]+2$  in the second variable (i.e., as function of $s$, $L_p(\kappa,[s])/\prod_{i=0}^h \log_p(u^{s-i}-1)$ is holomorphic on the open unit ball). For all classical points $(\kappa, \kappa')$, we have the following interpolation formula
\begin{align*}
 L_p(\kappa,\kappa') =  C_{\kappa,\kappa'} E_1(\kappa,\kappa')E_2(\kappa,\kappa') \frac{\Ll(s+1,\mathrm{Sym}^2(F(\kappa)),\xi^{-1}\eps^{-1}\omega^{-s})}{\pi^{s+1}S(F(\kappa))W'(F(\kappa))\lla F^{\circ}(\kappa),F^{\circ}(\kappa)\rra}.
\end{align*}
\item[ii)] 
The function $L_p^*(\kappa)$ is meromorphic on $\Cc_{F}$.  For $ k \geq k_0 -1$, we have the following interpolation formula
\begin{align*}
 L^*_p(\kappa) =  C_{\kappa} E_2(\kappa,[k_0-2]) \frac{\Ll(k_0-1,\mathrm{Sym}^2(F(\kappa)),{\xi'}^{-1})}{\pi^{k_0-1}S(F(\kappa))W'(F(\kappa))\lla F^{\circ}(\kappa),F^{\circ}(\kappa)\rra}.
\end{align*}
\end{itemize}
\end{theo}
If $\alpha=0$, using the direct estimate in \cite[Theorem 2.7.6]{DD}, we see that we can take $h=1$ and the first part of this theorem is \cite[Theorem]{H6}.\\
The poles on $\Cc_{F}$ of these two functions come from the poles of the overconvergent projection and of $l_{F}$; if $(\kappa,\kappa')$ corresponds to a couple of points which are classical, then locally around this point no poles appear. \\
Let us fix  a point $\kappa$ of $\Cc_F$ above $[k]$ and let $f$ be the corresponding form. If  $k > 2 \alpha+2$ then, by specializing at $\kappa$ the first variable, we recover the one variable $p$-adic $L$-function of the symmetric square of $f$ constructed in \cite{DD} (up to some Euler factors). If instead $k \leq 2 \alpha+2$, the method of \cite{DD} cannot give a well-defined one variable $p$-adic $L$-function because, as we said in the introduction, the Mellin transform of an $h$-admissible measure $\mu$ is well-defined only if the first $h+1$ moments are specified. But in this situation the number of critical integers is $k-1$ and consequently we do not have enough moments.  What we have to do is to choose the extra moments $\int \eps(u)u^{s} \textup{d}\mu$ for all $\eps$ finite-order character of $1+p\Z_p$ and $s=k-1,\ldots,h$.  We proceed as in \cite{PolSt}; the two-variable $p$-adic $L$-function $L_p(\kappa,\kappa')$ is well defined for all $(\kappa,\kappa')$, so we decide that 
\begin{align*}
L_p(s,\mathrm{Sym}^2(f),\xi):=L_p(\kappa,[s]).
\end{align*} 
This amounts to say that the extra moments for $\mu$ are 
\begin{align*}
\int \eps(u)u^{s} \textup{d}\mu = L_p(\kappa,\eps (\lla z \rra) z^s).
\end{align*} 
To justify such a choice, we remark that if a classical point $\kappa''$ of $\Cc_F$ is sufficiently close to $\kappa$, then $\kappa''$ is above $[k'']$ with $k''  > h+1$. In this case $L_p(\kappa'',\eps (\lla z \rra) z^s)$ interpolates the special values $L(s,\mathrm{Sym}^2(f'') \otimes \xi)$ which are critical \`a la Deligne. We are then choosing the extra moments by $p$-adic intepolation along the weight variable.
Fix $f$ as in Theorem \ref{MainThOC} and let $\kappa_0$ in $\U$ such that $F(\kappa_0)=f$. We have the following important corollary
\begin{coro}\label{CoroImp}
We have the following factorization of locally analytic functions around $\kappa_0$ in $\Cc_F$:
\begin{align*}
L_p(\kappa,[k_0-1])= (1 - \xi'(p)\lambda_p(\kappa)^{-2}p^{k_0-2})L_p^*(\kappa).
\end{align*}
\end{coro} 
We recall that this corollary is the key for the proof of Theorem \ref{MainThOC}.\\ 
The rest of the section will be devoted to the proof of Theorem \ref{T1OC}
\begin{proof}[Proof of Theorem \ref{T1OC}]
 Let $(\kappa,\kappa')$ be a classical point  as in the statement of the theorem. In particular $\kappa(u)=u^k$ and $\kappa'(u)=\eps(u)u^s$, with $0 \leq s \leq k-2$. \\
 We point out that all the calculations we need have already been performed in \cite{Pan,H6}. \\
If $\eps$ is not trivial at $p$, we shall write $p^n$ for the conductor of $\eps$. If $\eps$ is trivial, then we let $n=1$. Let $\beta=0,1$, $\beta \equiv s \bmod 2 $.\\
We have 
\begin{align*}
L_p(\kappa,\kappa') = & D^{-1} \frac{\lla F(\kappa)^c|\tau_{Np}, T_{D/N,k} U_p^{-2 n+1} \mathrm{Pr}^{\leq \alpha} U_p^{2 n-1} g   \rra}{\lla F(\kappa)^c | \tau, F(\kappa) \rra}, \\
g = &  \theta(\eps\xi\omega^{s})|[D/4C^2] E_{k-\frac{2\beta+1}{2}}(2k-s-\beta-3,\xi\psi\sigma_{-1}\omega^{-s}\eps)|\iota_p. 
\end{align*}
We have as in \cite[(7.11)]{Pan}
\begin{align*}
 \lla F(k)^c|\tau_{Np}, U_p^{-2 n +1} \mathrm{Pr}^{\leq \alpha} U_p^{2 n-1} g   \rra = \lambda_p(\kappa)^{1-2n}p^{(2n-1)(k-1)} \lla F(\kappa)^c|\tau_{Np}|[p^{2n-1}],  g   \rra, 
\end{align*}
where $f|[p^{2n-1}](z)=f(p^{2n-1}z)$. We recall the well-known formulae \cite[page 79]{H1bis}
\begin{align*}
 \lla f|[p^{2n}] , T_{D/N,k} g   \rra = & {(D/N)}^k \lla f|[(p^{2n}D)/N], g   \rra, \\
\tau_{Np}|[(p^{2n-1}D)/N] =  &{\left(\frac{p^{2n-1}D}{N}\right)}^{-k/2}\tau_{Dp^{2n}} ,\\
\frac{\lla F(\kappa)^{c}|\tau_{Np}, F(\kappa) \rra}{\lla F(\kappa)^{\circ}, F(\kappa)^{\circ} \rra}  = &(-1)^k W'(F(\kappa)) p^{(2-k)/2} \lambda_p(\kappa) \times \\
& \left( 1 - \frac{\psi(p)p^{k-1}}{\lambda_p(\kappa)^{2}} \right)\left( 1 - \frac{\psi(p)p^{k-2}}{\lambda_p(\kappa)^{2}} \right).
\end{align*}
Combining these with Lemma \ref{RankPet}, we have 
\begin{align*}
 L_p(\kappa,\kappa') = & D^{-1} i^k 2^{\frac{s+1+\beta}{2}+1-k} {(4\pi)}^{-\frac{s+1+\beta}{2}}{(2\pi)}^{-\frac{s+2-\beta}{2}} \G\left(\frac{s+1+\beta}{2} \right)\G\left(\frac{s+2-\beta}{2} \right)\times \\ 
  & \lambda_p(\kappa)^{1 -2n}p^{(2n-1)\left(\frac{k}{2}-1\right)} {(D/N)}^{k/2} {(Dp^{2n})}^{\frac{2s-2k +5}{4}} \times \\
 & \frac{D(\beta + s+1, f, \theta(\xi\omega^{s}\eps)|[D/(4C^2)]|\tau_{Dp^{2n}}) }{\lla F(\kappa)^{c}|\tau, F(\kappa)  \rra}.
\end{align*}
Let $\eta=\xi\omega^{s}\eps$. We recall from the transformation formula for theta series (see  \cite[(5.1 c)]{H6})
\begin{align*}
 \theta(\eta)|\tau_{4C^2p^{2n}}= \left\{
\begin{array}{cc}
(-i)^{\beta} {(Cp^{n})}^{-1/2} G(\eta) \theta(\eta^{-1}) & \mbox{ if } \eta \mbox{ primitive } \bmod p \\
\begin{array}{c}
-(-i)^{\beta}{(Cp)}^{-1/2} G(\eta) \eta_0(p) \times  \\
 \times (\theta(\eta_0^{-1}) - p^{\beta+1}\eta_0^{-1}(p)\theta(\eta_0^{-1})|[p^2]) 
\end{array} & \mbox{ if not. }
\end{array}\right.
\end{align*}
We have the following relations for weight $\frac{2\beta +1}{2}$:
\begin{align*}
 \tau_{Dp^n}=  \tau_{D}[p^n]p^{n{\frac{2\beta+1}{4}}}, \;\;\:\:  \left[D'\right]|_{\frac{2\beta +1}{2}}\tau_{D} =  \tau_{4C^2}  {(D')}^{-\frac{2\beta +1}{4}};
\end{align*}
 when $\eta$ is trivial modulo $p$ we obtain
\begin{align*}
 D(\beta + s, f, \theta(\eta)|\tau_{4C^2p^{2n}}) = & -(-i)^{\beta}{(Cp)}^{-1/2}  (1 -\lambda_p(\kappa)^2 p^{1-s} \eta_0^{-1}(p)) \times \\
 &  \times G(\eta) \eta_0(p) D(\beta +s, f, \theta(\eta_0^{-1})).
\end{align*}
We recall the well-known duplication formula
\begin{align*}
 \G(z)\G\left(z+\frac{1}{2} \right) = 2^{1-2z}\pi^{1/2}\G(2z).
\end{align*}
Summing up, we let $\delta=0$ resp. $\delta=1$ if $\eta_0$ has conductor divisible resp. not divisible by $p$. We obtain 
\begin{align*}
  L_p(\kappa,\kappa') = & D^{\frac{2s-2\beta}{4}} N^{-k/2} C^{\beta}  2^{-s -\frac{1}{2} -k} 2^{1-s -1} \times \\
& \times {(-1)}^{\delta} p^{sn}\lambda_p(\kappa)^{-2n}(1 -\lambda_p(\kappa)^2 p^{-s} \eta_0^{-1}(p))\eta_0(p) \times \\ 
 & \times E_2(\kappa,\kappa') \frac{\Ll(s+1,\mathrm{Sym}^2(F(\kappa)),\xi^{-1}\eps^{-1}\omega^{-s})}{\pi^{s+1}S(F(\kappa))W'(F(\kappa))\lla F^{\circ}(\kappa),F^{\circ}(\kappa)\rra}.
\end{align*}
We now evaluate  the  second $p$-adic $L$-function; we have
\begin{align*}
L_p^*(\kappa) = & \frac{\lla F(\kappa)^c|\tau_{Np}, T_{D/N,k} \mathrm{Pr}^{\leq \alpha}  g   \rra}{\lla F(\kappa)^c | \tau, F(\kappa) \rra}, \\
g = &  \theta(\xi') |\left[ \frac{D}{4C^2} \right] \delta_{k- k_0 - \frac{3}{2}}^{\frac{k_0 -\beta}{2}-1} E_{k- k_0 -1}(\sigma_{-1}\psi'\xi'). 
\end{align*}
The relation  
\begin{align*}
\theta(\xi') |\left[ \frac{D}{4C^2} \right]| \tau_{Dp}={\left(\frac{4C^2}{D}\right)}^{\frac{2\beta +1}{4}}  \theta(\xi')| \tau_{4C^2}[p]p^{\frac{2\beta+1}{4}}.
\end{align*}
gives us 
\begin{align*}
 D\left(\beta + k_0 -1, f, \theta(\xi')|\left[ \frac{D}{4C^2} \right]| \tau_{Dp}\right) = & p^{\frac{2\beta+1}{4}}\lambda_p(\kappa)p^{-\frac{k_0-1}{2}}{\left(\frac{4C^2}{D}\right)}^{\frac{2\beta +1}{4}}(-i)^{\beta} \\
 & \times {C}^{-1/2} G(\eta)  D(\beta + k_0-1, f, \theta({\xi'}^{-1})).
\end{align*}
\end{proof}

We now give  a proposition on the behavior of $L_p(\kappa,\kappa')$ along $\Delta(\kappa,\kappa')$. We say that $F(\kappa)$ has complex multiplication by a quadratic imaginary field $K$ if $F(\kappa)|T_l =0 $ for almost all $l$ inert in $K$. In particular, if $F(\kappa)$ has complex multiplication it is ordinary; indeed, all the non critical specializations are classical CM forms and a finite slope CM form of weight $k$ can only have slope $0$, $(k-1)/2$ or $k-1$.  As  we supposed that the slope of $F(\kappa)$ is fixed, then it must be zero.
We can prove exactly as \cite[Proposition 5.2]{H6} the following proposition
\begin{prop}\label{noCM}
Unless $\psi\xi\omega^{-1}$ is quadratic imaginary and $F(\kappa)$ has complex multiplication by the field corresponding to $\psi\xi\omega^{-1}$, $H(\kappa)L_p(\kappa,\kappa')$ is holomorphic at all points $(\kappa,\kappa')$ on the closed $\Delta(\kappa,\kappa')=0$, except possibly for a finite number of points of type $([k],[k-1])$ corresponding to CM forms.
\end{prop}
 Suppose that the family $F(\kappa)$ specializes to a critical CM  form.  It is known that it is in the image of the operator $\Theta^{k-1}$ and therefore should correspond to a zero of $H_F(\kappa)$. Moreover such a point on $\Cc_F$ is ramified above the weight space \cite[Proposition 1]{BelCM}. 

\section{The proof of Benois' conjecture}\label{Benconj}
In this section we shall prove a more general version of Theorem \ref{MainThOC}. Once one knows Corollary \ref{CoroImp}, what he is left to do is to reproduce {\it mutatis mutandis} the method of Greenberg-Stevens. We remark that we have a shift $s \mapsto s-1$ between the $p$-adic $L$-function of the previous section and the one of the introduction.\\
From the interpolation formula given in Theorem \ref{T1OC}, we see that we have a trivial zero when $s=k_0-2$ and $(\omega^{2-k_0}\xi)_0(p)=1$. Let us denote by $\Ll(f)$ the $\Ll$-invariant of $\mathrm{Sym}^2(f)\otimes \xi( k_0 -1)$ as defined in \cite{BenLinv}. We have the following theorem
\begin{theo}\label{T2}
Fix $f$ in $\M_{k_0}(Np,\psi)$ and suppose that $f$ is Steinberg at $p$. Let $\xi=\xi'\omega^{k_0-2}$ be a character  such that $\xi(-1)={(-1)}^{k_0}$ and $\xi'(p)=1$. Then 
\begin{align*}
 \lim_{s \to k_0 -1} \frac{L_p(s,\mathrm{Sym}^2(f), \xi)}{s-k_0+1} = \Ll(f) \frac{\Ll(k_0-1,\mathrm{Sym}^2(f),\xi^{-1})}{\pi^{k_0-1} W'(f)S(f)\Omega(f)}.
\end{align*}
\end{theo}
We shall leave the proof of this theorem for the end of the section. We now give the proof of the main theorem of the paper.\\
\begin{proof}[Proof of Theorem  \ref{MainThOC}]
 We let $\xi'=\mathbf{1}$. Assuming Theorem \ref{T2}, what we are left to show is that $\Ll(s,\mathrm{Sym}^2(f))$ coincides with the completed $L$-function $L(s,\mathrm{Sym}^2(f))$.\\ 
As we said in Section \ref{primLfun}, the two $L$-functions differ only for some Euler factors at primes dividing $2N$.\\
As $2 \mid N$, we have $(1-\psi^2(2))=1$ as $\psi(2)=0$.\\ 
We have seen in Section \ref{primLfun} that when $\pi(f)_q$ is a Steinberg representation, the Euler factors at $q$ of $\Ll(s,\mathrm{Sym}^2(f))$ and $L(s,\mathrm{Sym}^2(f))$ are the same.\\
As the form $f$ has trivial Nebentypus and squarefree conductor, we have that $\pi(f)_q$ is Steinberg for all $q \mid N$ and we are done.
\end{proof}
More precisely, we have that Theorem \ref{T2} implies Conjecture \ref{MainCoOC} any time that the factor ${\calE_N(k-1,f,\xi)}^{-1}$ is not zero, for the same reasoning as above.  This is true if, for example, the character $\xi$ is very ramified modulo $2N$.\\
If we choose  $\xi=\psi^{-1}$, we are then considering the $L$-function for the representation $\mathrm{Ad}(\rho_{f})$. 
In this case the conditions for ${\calE_N(k-1,f,\xi)}^{-1}$ to be non-zero are quite restrictive. For example $2$ must divide the level of $f$. If moreover we have that the weight is odd, then there exist at least a prime $q$ for which, in the notation of Section \ref{primLfun}, $\pi_q$ is a ramified principal series. From the explicit description of the Euler factors at $q$ given in Section \ref{primLfun}, we see that ${L_q(0,\hat{\pi}(f))}^{-1}$ is always zero.  \\

The $\Ll$-invariant for the adjoint representation has been calculated in the ordinary case in \cite{HLinv}. This approach has been generalized to calculate Benois' $\Ll$-invariant in \cite{BenLinv2,MokLinv}. These results can be subsumed as follows
\begin{theo}\label{L-inv}
 Let $F(\kappa)$ be a family of overconvergent eigenforms such that $F(\kappa_0)=f$ and let $\lambda_p(\kappa)$ be its $U_p$ eigenvalue. We have 
\begin{align*}
 \Ll(f) = & -2 \frac{\textup{d} \log \lambda_p(\kappa)}{\textup{d}\kappa}\vert_{\kappa=\kappa_0}.
\end{align*}
\end{theo}
We remark that it is very hard to determine whether  $\Ll(f)$ is not zero, even though the above theorem tells us that this is always true except for a  finite number of points.\\ 
If we suppose $k_0=2$, we are considering an ordinary form and in this case $\rho_f\vert_{\Q_p}$ is an extension of $\Q_p$ by $\Q_p(1)$. The $\Ll$-invariant can be described via Kummer theory. Let us denote by $q_f$ the universal norm associated to the extension $\rho_f\vert_{\Q_p}$;  we have then $ \Ll(f)= \frac{\log_p(q_f)}{\mathrm{ord}_p(q_f)}$. Let $A_f$ be the abelian variety associated to $f$, in \cite[\S 3]{SSS} the two authors give a description of $q_f$ in term of the $p$-adic uniformization of $A_f$. When $A_f$ is an elliptic curve, then $q_f$ is Tate's uniformizer and a theorem of transcendental number theory \cite{Saint} tells us that $\log_p(q_f)\neq 0$.

\begin{proof}[Proof of Theorem \ref{T2}]
Let $\kappa_0$ be the point on $\Cc$ corresponding to $f$. As the weight of $f$ is not critical, we have that  $w:\Cc \rightarrow \W$ is \'etale at $\kappa_0$. We have $w(\kappa_0)=[k_0]$. Let us write $t_0 = (z \mapsto \omega^{-k_0}(z)z^{k_0} )$; $t_0$ is a local uniformizer in $\mathcal{O}_{\W,[k_0-1]}$. As the map $w$ is \'etale at $\kappa_0$, $t_0$ is a local uniformizer for $\mathcal{O}_{\Cc,\kappa_0}$. Let us write $A=\mathcal{O}_{\Cc,\kappa_0}/(T_0^2)$, for $T_0=\kappa - t_0$. We have an isomorphism between the tangent spaces; this induces an isomorphism on derivations
\begin{align*}\
\mathrm{Der}_{K}(\mathcal{O}_{\W,[k_0-1]},\C_p) \cong \mathrm{Der}_{K}(\mathcal{O}_{\Cc,\kappa_0},\C_p).
\end{align*}
The isomorphism is made explict by fixing a common basis $\frac{\partial }{\partial T_0}$.\\
We take the local parameter at $[k_0-2]$ in $\W$ to be  $t_1=(z \mapsto \omega^{-k_0+2}(z)z^{k_0-2} )$.\\
Let $\xi$ be  as in the hypothesis of the theorem and let $L_p(\kappa,\kappa')$ be the $p$-adic $L$-function constructed in Theorem \ref{T1OC}. 
We can see, locally  at $(\kappa_0,[k_0-2])$, the two variables $p$-adic $L$-function $L_p(\kappa,\kappa')$  as a function $L_p(t_0,t_1)$ of the two local parameters $(t_0,t_1)$. Let us define $t_0(k)=(z \mapsto \omega^{-k_0}(z)z^{k} )$  resp. $t_1(s)=(z \mapsto \omega^{-k_0+2}(z)z^{s} )$ for $k$ resp. $s$ $p$-adically close to $k_0$ resp. $k_0 -2$. Consequently, we pose $L_p(k,s)=L_p(t_0(k),t_1(s))$; this is a locally analytic function around $(k_0,k_0-2)$.\\
We have $ \frac{\partial }{\partial k}=\log_p(u) \frac{\partial }{\partial \log T_0}$. The interpolation formula of Theorem \ref{T1OC} {\it i)} tells us that locally $L(k,k-2)\equiv 0$.  We derive this identity with respect to $k$ to obtain
\begin{align*}
 \left. \frac{\partial L(k,s)}{\partial s } \right\vert_{s=k_0-1,k=k_0} =- \left.\frac{\partial L(k,s)}{\partial k } \right\vert_{s=k_0-1,k=k_0}.
\end{align*}
Using Corollary \ref{CoroImp} and Theorem \ref{L-inv} we see that 
\begin{align*}
 L_p(k,k_0-1) =\Ll(f) L_p^*(k_0) + O({(k-k_0)}^2)
\end{align*}
and we can conclude thanks to the second interpolation formula of Theorem \ref{T1OC}.
\end{proof}

\section{Relation with other symmetric square $p$-adic $L$-functions}\label{KimBel}
As we have already said, the $p$-adic $L$-function of the previous section $L_p(\kappa,\kappa')$ has some poles on $\Cc_F$ coming from the {\it ``denominator''} $H_F(\kappa)$ of $l_F^r$. In this section we will see how we can modify it to obtain a holomorphic function using a one-variable $p$-adic $L$-function for the symmetric square constructed by Kim and, more recently, Bella\"iche. The modification we will perform to $L_p(\kappa,\kappa')$ will also change the interpolation formula of Theorem \ref{T1OC}, changing the automorphic period (the Petersson norm of $f$) with a motivic one. We shall explain in the end of the section why, in the ordinary case, this change is important for the Greenberg-Iwasawa-Coates-Schmidt Main Conjecture \cite[Conjecture 1.3.4]{Urb}.\\
We want to point out that all we will say in this section is conditional on Kim's thesis \cite{Kim} which has not been published yet.\\
In the ordinary setting, the overconvergent projector and the ordinary projector coincide and it is known in many cases, thanks to Hida \cite[Theorem 0.1]{H5}, that $H_F(\kappa)$ interpolates, up to a $p$-adic unit, the special value 
\begin{align*}
(k-1)! W'(F({\kappa})) E^*_3(\kappa)\frac{L(k,\mathrm{Sym^2}(F({\kappa})),\psi^{-1})}{\pi^{k+1}\Omega^+\Omega^-},
\end{align*} 
where $W'(F({\kappa}))$ is the root number of $F^{\circ}(\kappa)$ and  $E_3^*(\kappa)=1$ if $F({\kappa})$ is primitive at $p$ and 
\begin{align*}
\left( 1 - \frac{\psi(p)p^{k-1}}{\lambda_p(\kappa)^{2}} \right)\left( 1 - \frac{\psi(p)p^{k-2}}{\lambda_p(\kappa)^{2}} \right)
\end{align*}
otherwise (note that, up to a sign, it coincides with $S(F(\kappa))$). Here $\Omega^+=\Omega^+(F(\kappa))$ and $\Omega^-=\Omega^-(F(\kappa))$ are two complex periods defined via the Eichler-Shimura isomorphism.\\ 
Kim in his thesis \cite{Kim} and recently Bella\"iche  generalize Hida's construction to obtain a one variable $p$-adic $L$-function for the symmetric square. The aim of this section is to confront the $p$-adic $L$-function of section \ref{padicL} with theirs.\\
Kim's idea is very beautiful and at the same time quite simple; we will sketch it now. Its construction relies on two key ingredients; the first one is the formula, due to Shimura, \begin{align}\label{PetSym}
 \lla f^{\circ}, f^{\circ} \rra =(k-1)! \frac{L_N(k,\mathrm{Sym}^2(f),\psi^{-1})}{N^2 2^{2k_0}\pi^{k_0+1}}.
\end{align}
The second one is the sheaf over the eigencurve $\Cc$ of distribution-valued modular symbol, sometimes  called the {\it overconvergent modular symbol}, constructed by Stevens \cite{PolSt,BelCri}. It is a sheaf interpolating the sheaves $\mathrm{Sym}^{k-2}(\Z_p^2)$ appearing in the classical Eichler-Shimura isomorphism. When modular forms are seen as sections on this sheaf, the Petersson product is induced by the natural pairing on $\mathrm{Sym}^{k-2}(\Z_p^2)$. Kim's idea is to interpolate these pairings when $k$ varies in the weight space to construct a pairing on the space of locally analytic functions on $\Z_p$. This will induce a (non-perfect) pairing $\lla \phantom{e}, \phantom{e} \rra_{\kappa} $ on the sheaf of overconvergent modular symbol. For a family $F(\kappa)$ we can define two modular symbols $\Phi^{\pm}(F(\kappa))$. Kim defines 
\begin{align*}
L_p^{KB}(\kappa)= \lla \Phi^{+}(F(\kappa)), \Phi^{-}(F(\kappa)) \rra_{\kappa}.
\end{align*}
This $p$-adic $L$-function satisfies the property that its zero locus contains the ramification points of the map $\Cc\rightarrow \W$ \cite[Theorem 1.3.3]{Kim} and for all classical non critical point $\kappa$ of weight $k$ we have the following interpolation formula
\cite[Theorem 3.3.9]{Kim}
\begin{align*}
L_p^{KB}(\kappa)= E^*_3(\kappa)W'(F(\kappa))\frac{ (k-1)! L_N(k,\mathrm{Sym}^2(F(\kappa)),\psi^{-1})}{N^2 2^{2k}\pi^{k+1}\Omega^+\Omega^-}.
\end{align*}
The period $\Omega^+\Omega^-$ is the one predicted by Deligne's conjecture for the symmetric square motive and it is probably a better choice than the Petersson norm of $f$ for at least two reasons. 
The first one is that, as we have seen in (\ref{PetSym}), the Petersson norm of $f$ essentially coincides with $L(k,\mathrm{Sym}^2,\psi^{-1})$ and such a choice as a period is not particulary enlightening when one is interested in Bloch-Kato style conjectures. \\
The second reason is related to the the Main Conjecture. Under certain hypotheses, such a conjecture  is proven  in \cite{Urb} for the $p$-adic $L$-function with  motivic period. In fact, in \cite[\S 1.3.2]{Urb} the author is forced to make a change of periods from the $p$-adic $L$-function of \cite{CS,H6,DD} to obtain equality of $\mu$-invariant.\\ 
It seems reasonable to the author that in many cases, away from the zero of the overconvergent projection, we could choose $H_F(\kappa)=L_p^{KB}(\kappa)$; in any case, we can define a function 
\begin{align*}
\tilde{L}_p(\kappa,\kappa'):=L_p^{KB}(\kappa)L_p(\kappa,\kappa')
\end{align*}
which is locally  holomorphic in $\kappa$ and at classical points interpolates, up to some explicit algebraic number which we do not write down explicitly, the special values $\frac{ \Ll(s,\mathrm{Sym}^2(F(\kappa)),\eps^{-1}\omega^{s-1})}{ \pi^{s}\Omega^+\Omega^-}$.

\appendix
\section{Functional equation and holomorphy}\label{FE}
The aim of this appendix is to show that we can divide the two-variable $p$-adic $L$-function constructed in Section \ref{padicL} by suitable two-variable functions to obtain a holomorphic $p$-adic $L$-function interpolating the special values of the primitive $L$-function, as defined in Section \ref{primLfun}.\\
The method of proof follows closely the one used in \cite{DD} and \cite{H6}. We shall first construct another two variables $p$-adic $L$-function, interpolating the other set of critical values. The construction of this two variables $p$-adic $L$-function has its own interest. The missing Euler factors do not vanish, and if one could prove a formula for the derivative of this function would obtain a proof of Conjecture \ref{MainCoOC} without hypotheses on the conductor.\\
We will show that, after dividing by suitable functions, this  $p$-adic $L$-function and the one of Section \ref{padicL} satisfy a functional equation. We shall conclude by showing that the poles of these two functions are distinct.\\
We start recalling the  Fourier expansion of some Eisenstein series from \cite[Proposition 3.3.10]{RosTh};  
\begin{align*}
E_{k-\frac{1}{2}}(z,0;\chi) = & L_{Lp}(2k-3,\chi^2) + \sum_{n=1}^{\infty} q^n L_{Dp}\left(k-1,\chi\sigma_n\right) \times  \\ 
&  
\times \left(\sum_{\tiny{\begin{array}{c} t_1^2t_2^2 |n, \\ (t_1t_2,Dp)=1, \\ t_1>0, t_2>0 \end{array}}} \mu(t_1)\chi(t_1t_2^2)\sigma_n(t_1)t_2{(t_1t_2^2)}^{1-k}\right).
\end{align*}
We have for $0 \leq s \leq k/2$
\begin{align*}
 \delta_{k+\frac{1}{2}}^{s}E_{k+\frac{1}{2}}(z,0,\chi) = E_{k + 2s + \frac{1}{2}}(z,2s;\chi). 
\end{align*}
We give this well known lemma;
\begin{lemma}\label{zetames+}
Let $\chi$ be a even primitive character modulo $Cp^r$, with $C$ and $p$ coprime and $r\geq 0$. Then for any $b\geq 2$ coprime with $p$, there exists a measure $\zeta^+_{\chi,b}$ such that for every finite-order character $\eps$ of $Z_D$ and any integer $m\geq 1$ we have
\begin{align*}
\int_{Z_D} \eps(z)z_p^{m-1}\textup{d}\zeta^+_{\chi,b}(z)=(1-\eps'\chi'(b)b^{m})(1-(\eps\chi)_0(p)p^{m-1})\times \\ \times \frac{G((\eps\chi)_p)}{p^{(1-m) c_p}}\frac{L_{D}(m,\chi^{-1}\eps^{-1})}{\Omega(m)},
\end{align*}
where $\chi'$ denote the prime-to-$p$ part of $\chi$, $\chi_p$ the $p$-part of $\chi$ and $c_p$ the $p$-part of the conductor of $\chi\eps$. 
If we let $a=0$, $1$ such that $\eps\chi(-1)={(-1)}^a$, we have 
\begin{align*}
{\Omega(m)}^{-1}= {\Omega(m,\eps\chi)}^{-1} = i^a \pi^{1/2-m}\frac{\G(\frac{m+a}{2})}{\G(\frac{1-m+a}{2})}. 
\end{align*}
\end{lemma}
As before, we can associate to this measure a formal series  
\begin{align*}
G^+(S,\xi,\chi,b) = \int_{Z_D} \xi(z) {(1+S)}^{z_p} \textup{d}\zeta^+_{\chi,b}(z).
\end{align*}
We shall denote by $L^+_p(\kappa,\xi,\chi,b)$ the image of $G^+(S,\xi,\chi,b)$ by the map $S\mapsto (\kappa \mapsto \kappa(u)-1)$.\\
We define an element of $\A(\W)[[q]]$
\begin{align*}
\calE^+_{\kappa}(\chi)  & = \sum_{n=1, (n,p)=1}^{\infty} L^+_p(\kappa[-2],\chi, \sigma_n,b) 
q^n \sum_{\tiny{\begin{array}{c} t_1^2t_2^2 |n, \\ (t_1t_2,Dp)=1, \\ t_1>0, t_2>0 \end{array}}}t_1^{2}t_2^{3} \mu(t_1)\chi(t_1t_2^2)\sigma_n(t_1) \kappa^{-1}({t_1t_2^2}).
\end{align*}
If $\kappa=[k]$, we have then 
\begin{align*}
[k](\calE^+_{\kappa}(\chi))=&\frac{G((\chi)_p)}{p^{(2-k) c_p}\Omega(k-1)}(1-\chi'(b)b^{k-1})E_{k -\frac{1}{2}}(z,0,\omega^k \chi^{-1})|\nu_k \\ 
\nu_k(n)=&\frac{(1-\omega^{\frac{p-1}{2}}(n)(\chi\omega^k)_0(p)p^{k-3})}{(1-\omega^{\frac{p-1}{2}}(n)(\chi^{-1}\omega^{-k})_0(p)p^{2-k})},
\end{align*}
where the twist by $\nu_k$ is defined as in \cite[h5]{H6}.
We fix two even Dirichlet characters as in Section \ref{padicL}: $\xi$ is primitive modulo ${\Z/Cp^{\delta}\Z}$ ($\delta=0,1$) and $\psi$ is defined modulo ${\Z/pN\Z}$. Fix also a positive slope $\alpha$ and a positive integer $D$ which is a square and divisible by $C^2$, $4$ and $N$. Let us denote by $C_0$ the conductor of the prime-to-$p$ part of $\xi\psi^{-2}$ and let us write $D=4C_0^2D_0'$.\\
For $s=0,1,\ldots$ we now define  distributions $\mu^+_{s}$ on $\Z^{\times}_p$ with values in  ${\N^{r}(D,\A(\W))}^{\leq \alpha}$. For any $\eps$ of conductor $p^n$ we pose 
\begin{align*}
 \mu^+_{s}(\eps) = \mathrm{Pr}^{\leq{\alpha}}U_p^{2n-1}\left(\theta(\psi^2\xi^{-1} \eps^{-1}\omega^{-s})|\left[\frac{D}{4C_0^2}\right]\delta_{\kappa\left[-s-\frac{1}{2}\right] }^{\frac{s-\beta}{2}} \calE^+_{\kappa[-s]}(\psi\xi^{-1} \eps^{-1} \sigma_{-1})\right)
\end{align*}
 with $\beta =0$, $1$ such that $s\equiv \beta \bmod 2$. 
\begin{prop}
The distributions $\mu^+_s$ define an $h$-admissible measure $\mu^+$ with values in ${\N^{r}(D,\A(\W(t_0) \times \W))}^{\leq \alpha}$ (for $t_0$ as before Proposition \ref{GlueDist}).
\end{prop}
We take the Mellin transform 
\begin{align*}
 \kappa' \mapsto \int_{1+p\Z_p} {\kappa'(u)}^{z}\textup{d}\mu^+(z)
\end{align*}
  to obtain an element $\theta\ast E^+(\kappa,\kappa')$ of ${\N^{r}(D,\A(\W(t_0) \times \W(t)))}^{\leq \alpha}$.\\
Let $F$ be a family of finite slope eigenforms. We refer to Section \ref{padicL} for all the unexplained notation and terminology. We pose
\begin{align*}
\Delta(\kappa,\kappa') = & \left(1-\psi'{\xi'}^{-1}(b)\frac{\kappa(u)}{b\kappa'(u)}\right).
\end{align*}
We define a new  $p$-adic $L$-function 
\begin{align*}
L^+_p(\kappa,\kappa')= & D^{-1}{\Delta(\kappa,\kappa')}^{-1} l_{F}(T_{D/N} \theta \ast E^+(\kappa,\kappa')) \in \K(\Cc_F \times \W).
\end{align*}
Let us define
\begin{align*} 
E^+_1(\kappa,\kappa')= & \lambda_p(\kappa)^{-2n}(1-(\xi^{-1}\eps^{-1}\omega^{-s}\psi^{2})_0(p)\lambda_p(\kappa)^{-2} p^{2k-3-s});
\end{align*}
when $F(\kappa)$ is primitive at $p$ we define $E_2^+(\kappa,\kappa')=1$, otherwise
\begin{align*}
E^+_2(\kappa,\kappa')= & (1-\xi^{-1}\eps^{-1}\omega^{-s}\psi(p)p^{k-2-s})  (1-{(\xi\eps\omega^{s})}_0(p)\lambda_p(\kappa)^{-2}p^{s}).
\end{align*}
Let $\beta=0$, $1$ such that $s \equiv \beta \bmod 2$, we pose 
\begin{align*}
C^+_{\kappa,\kappa'}   = & (2k-3-s)! p^{n(3k-2s-5)}G(\psi \xi^{-1}\eps^{-1}\omega^{-s})G(\psi^2 \xi^{-1}\eps^{-1}\omega^{-s}) \times \\  
& \times C_0^{2k-s-1} N^{-k/2} {D_0'}^{k-1-\frac{s+\beta+1}{2}} 2^{2s+5-5k+\frac{1}{2}}.
\end{align*}
\begin{theo}\label{T3}
 $L^+_p(\kappa,\kappa')$ is a function on $\Cc_{F} \times \W$, meromorphic in the first variable and of logarithmic growth $h=[2 \alpha]+2$  in the second variable. For all classical points $(\kappa, \kappa')$, we have the following interpolation formula
\begin{align*}
 L^+_p(\kappa,\kappa') =  C^+_{\kappa,\kappa'}  \frac{E^+_1(\kappa,\kappa')E^+_2(\kappa,\kappa')\Ll(2k-2-s,\mathrm{Sym}^2(F(\kappa)),\xi\eps\omega^s)}{\Omega(k-s-1)\pi^{2k-s}S(F(\kappa))W'(F(\kappa))\lla F^{\circ}(\kappa),F^{\circ}(\kappa)\rra}.
\end{align*}
\end{theo}
\begin{proof}
The calculation are essentially the same as Theorem \ref{T1OC}; the only real difference is the presence of the twist by $\nu_k$. We can deal with it as we did in \cite[Theorem 3.11.2]{RosTh} so we shall only sketch the calculations. We first remark the following; let $\chi$ be any character modulo $p^r$, then it is immediate to see the following identity of $q$-expansions
\begin{align*}
U_{p^r}\left( \sum_n \chi(n)a_n q^n \sum_m a_m q^m \right) = \chi(-1)U_{p^r}\left( \sum_n a_n q^n \sum_m \chi(m)a_m q^m \right).
\end{align*}
We can write 
\begin{align*}
\frac{1}{(1-\omega^{\frac{p-1}{2}}(n)(\chi)_0(p)p^{k-2})}= 1 + \omega^{\frac{p-1}{2}}(n)(\chi^{-1}\omega^{k})_0(p)p^{k-2} + \dots . 
\end{align*}
We apply this to \begin{align*}
\theta(\psi^2\xi^{-1} \eps^{-1}\omega^{-s})|\left[\frac{D}{4C_0^2}\right] [k]\delta_{\kappa\left[-s-\frac{1}{2}\right] }^{\frac{s-\beta}{2}} \calE^+_{\kappa[-s]}(\psi\xi^{-1} \eps^{-1} \sigma_{-1})
\end{align*}
and we see that we can move the twist $\nu_k$ to $\theta(\psi^2\xi^{-1} \eps^{-1}\omega^{-s})|\left[\frac{D}{4C^2}\right]$, and we conclude noticing that 
\begin{align*}
\nu_k\left(\frac{D}{4C^2}n^2\right)=\frac{(1-(\psi\xi^{-1} \eps^{-1} \sigma_{-1})_0(p)p^{k-s-2})}{(1-(\psi^{-1}\xi \eps \sigma_{-1})_0(p)p^{s+1-k})}
\end{align*}
is independent of $n$. We have  
\begin{align*}
L^+_p(\kappa,\kappa') = & i^k C_{s-\beta,k-\beta} G(\eta_0^{-1})D^{-1}{(D/N)}^{k/2} {D_0'}^{-\frac{2\beta+1}{4}}(-i)^{\beta}{(C_0p^n)}^{-1/2} \frac{G(\psi\xi^{-1}\eps^{-1}\omega^{-s})}{p^{n(1-k+s)} \Omega(k-s-1)} \times \\
& \times p^{-(2n-1)\frac{k}{2}}  (1 -\lambda_p(\kappa)^2 p^{s-2k-2} \eta_0^{-1}(p))  \lambda_p(\kappa)^{1-2n}p^{(2n-1)(k-1)} \frac{\lla F(\kappa)^c,  g   \rra}{\lla F(\kappa)^c | \tau, F(\kappa) \rra}, \\
g = & 2^{\frac{\beta-s}{2}} \theta(\psi^{-2}\xi\eps\omega^{s})  E^*_{k-\frac{2\beta+1}{2}}(\beta-s,\psi^{-1}\xi\eps\omega^{s}\sigma_{-1})y^{\frac{\beta-s}{2}},\\
C_{s-\beta,k-\beta} = & {(2\pi)}^{\frac{s-2k+1+\beta}{2}} {(Dp^{2n})}^{\frac{2k  -2s-1}{4}}\G\left(\frac{2k-s-1-\beta}{2} \right). 
\end{align*}
We recall the well-known duplication formula
\begin{align*}
 \G(z)\G\left(z+\frac{1}{2} \right) = 2^{1-2z}\pi^{1/2}\G(2z)
\end{align*}
which we apply for $z=\frac{2k-s-2}{2}$. 
Summing up, we obtain 
\begin{align*}
  L_p^+(\kappa,\kappa') = &  (2k-3-s)! 2^{2s+5-5k+\frac{1}{2}} N^{-k/2}C_0^{2k-s-1}{D'_0}^{k-1-\frac{s+\beta+1}{2}} \times\\ 
  & \times G(\psi\xi^{-1}\eps^{-1}\omega^{-s})G(\psi^2\xi^{-1}\eps^{-1}\omega^{-s}) p^{n(3k-2s-5)} \lambda_p(\kappa)^{-2n} \times\\ 
 & \times \frac{(1 -\eta_0^{-1}(p)\lambda_p(\kappa)^2 p^{s-2k-2}) E^+_2(\kappa,\kappa') \Ll(2k-2-s,\mathrm{Sym}^2(F(\kappa)),\psi^2\xi\eps\omega^{s})}{\pi^{2k-2-s}S(F(\kappa))W'(F(\kappa))\lla F^{\circ}(\kappa),F^{\circ}(\kappa)\rra \Omega(k-s-1)}.
\end{align*}
\end{proof}
To interpolate the primitive $L$-function we have to divide $L_p(\kappa,\kappa')$ by some functions which interpolate the extra factors given in Section \ref{primLfun}. Let $F(\kappa)$ be as above and let us denote by $\left\{\lambda_n(\kappa) \right\}$ the corresponding system of Hecke eigenvalues. For any Dirichlet character of  prime-to-$p$ conductor $\chi$, let us denote by $F_{\chi}(\kappa)$ the primitive family of eigenforms associated to the system of Hecke eigenvalues $\left\{\lambda_n(\kappa)\chi(n) \right\}$.  
Let $q$ be a prime number and $f$ a classical modular form, we say that $f$ is minimal at $q$ if the local representation $\pi(f)_q$ has minimal conductor among its twists. Let $\chi$ be a Dirichlet character such that $F_{\chi}(\kappa)$ is minimal everywhere for every non critical point $\kappa$. As the Hecke algebra $\T^r(N,\K(\U))$ is generated by a finite number of Hecke operators, if $K$ is big enough to contain the values of $\chi$, then the Hecke eigenvalues of $F(\kappa)$ and $F_{\chi}(\kappa)$ all belong to $\A(\U)$. We shall denote by $\lambda^{\circ}_q(\kappa)$ the Hecke eigenvalue corresponding to the family which is minimal at $q$ and by $\alpha_q^{\circ}(\kappa)$ and $\beta_q^{\circ}(\kappa)$ the two roots of the corresponding Hecke polynomial; enlarging $\A(\U)$ if necessary, we can suppose that both of them belong to  $\A(\U)$.\\
For each prime $q$, let us write $l_q=\frac{\log_p(q)}{\log_p(u)}$. Recall the partition of the primes dividing  the level of $F$ {\it i), \ldots, iv)} given in Section \ref{primLfun}, we define
\begin{align*}
E_{q}(\kappa,\kappa') = & {\left(1-\xi^{-1}_0(q)q^{-1}\alpha^{\circ}_q(\kappa)^2 \kappa'(u^{-l_q})\right)}^{-1} \times \\
& \times {\left(1-(\psi\xi^{-1})_0(q)q^{-2}\frac{\kappa(u^{l_q})}{\kappa'(u^{l_q})}\right)}^{-1}
{\left(1-\xi_0^{-1}(q)q^{-1}\beta^{\circ}_q(\kappa)^{2}{\kappa'(u^{-l_q})}\right)}^{-1}  \;\: \left(\mbox{if } q \mbox{ in case } i\right),\\
E_{q}(\kappa,\kappa') = & {\left(1-(\psi\xi^{-1})_0(q)q^{-2}\frac{\kappa(u^{l_q})}{\kappa'(u^{l_q})}\right)}^{-1} {\left(1-(\psi^2\xi^{-1})_0(q)q^{-1}\lambda^{\circ}_q(\kappa)^{-2}{\kappa'(u^{l_q})}\right)}^{-1} \;\: \left(\mbox{if } q \mbox{ in case } ii\right),\\
E_{q}(\kappa,\kappa') = & \prod_{j \mbox{ s.t. } \pi_q \cong \pi_q \otimes \lambda_j} {\left(1-(\psi\lambda_j\xi^{-1})_0(q)q^{-2}\frac{\kappa(u^{l_q})}{\kappa'(u^{l_q})}\right)}^{-1} \;\: \left(\mbox{if } q \mbox{ in case } iv\right)
\end{align*}
and we pose
\begin{align*}
A(\kappa,\kappa')=& {\left(1-\psi^{-2}\xi^{2}(2)2^{2} \frac{\kappa'(u^{2 l_2})}{\kappa(u^{2 l_2})}\right)}\prod_{q}{E_{q}(\kappa,\kappa')}^{-1}.
\end{align*}
We define also 
\begin{align*}
E^+_{q}(\kappa,\kappa') = & {\left(1-(\psi^{-2}\xi)_0(q)q^{2}\alpha^{\circ}_q(\kappa)^2\frac{\kappa'(u^{l_q})}{\kappa(u^{2l_q})}\right)}^{-1}\times \\
& \times {\left(1-(\psi^{-1}\xi)_0(q)q \frac{\kappa'(u^{l_q})}{\kappa(u^{l_q})}\right)}^{-1}
{\left(1-(\psi^{-2}\xi)_0(q)q^{2}\beta^{\circ}_q(\kappa)^{2}\frac{\kappa'(u^{l_q})}{\kappa(u^{2l_q})}\right)}^{-1}\;\: \left(\mbox{if } q \mbox{ in case } i\right),\\
E^+_{q}(\kappa,\kappa') = & {\left(1-(\psi^{-1}\xi)_0(q)q\frac{\kappa'(u^{l_q})}{\kappa(u^{l_q})}\right)}^{-1} {\left(1-(\psi^{-1}\xi)_0(q)q^{-1}\lambda^{\circ}_q(\kappa)^{-2}\kappa(u^{l_q}){\kappa'(u^{l_q})}\right)}^{-1}\;\: \left(\mbox{if } q \mbox{ in case } ii\right),\\
E^+_{q}(\kappa,\kappa') = & \prod_{j \mbox{ s.t. } \pi_q \cong \pi_q \otimes \lambda_j} {\left(1-(\psi^{-1}\lambda_j\xi)_0(q)q\frac{\kappa'(u^{l_q})}{\kappa(u^{l_q})}\right)}^{-1} \;\: \left(\mbox{if } q \mbox{ in case } iv\right)
\end{align*} and we pose 
\begin{align*}
B(\kappa,\kappa')=& {\left(1-\psi^2\xi^{-2}(2)2^{-4} \frac{\kappa(u^{2 l_2})}{\kappa'(u^{2 l_2})}\right)}\prod_{ q }{E^+_{q}(\kappa,\kappa')}^{-1} .
\end{align*}
\begin{prop}\label{padicFE}
We have the following equality of meromorphic functions on $\Cc_F \times \W$ 
\begin{align*}
L_p(\kappa,\kappa'){A(\kappa,\kappa')}^{-1}= &  \beps(\kappa,\kappa') L^+_p(\kappa,\kappa'){B(\kappa,\kappa')}^{-1}
\end{align*}
where $\beps(\kappa,\kappa')$ is the only Iwasawa function such that 
\begin{align*}
\beps(u^k,u^s)= \frac{G({\chi'}^{-1}){G({\xi'}^{-1}\psi')}^2 {G({\xi'}^{-1}{\psi'}^2)}^2}{G(\xi')}{D'_0}^{\frac{s+1+\beta}{2}+1 - k}{D'}^{\frac{s-\beta}{2}} C'(\pi\otimes\xi)^{s-k+1} 2^{4k-4s-6} ,
\end{align*} for $C'(\pi\otimes\xi)$ the conductor outside $p$ of $\hat{\pi}\otimes \psi \xi^{-1}$.
\end{prop}
\begin{proof}
The explicit epsilon factor of the functional equation stated in Section \ref{primLfun} can be found in \cite[Theorem 1.3.2]{DD}.\\
Recall from {\it loc. cit.} that
\begin{align*}
\frac{L_{\infty}(s+1)}{L_{\infty}(2k-2-s)} &= \frac{s!{(2 \pi)}^{-s-1}\pi^{-\frac{s+1}{2}} \G\left(\frac{s-k+2+a}{2}\right)}{(2k-3-s)!{(2 \pi)}^{-2k+2+s}\pi^{-\frac{2k-s-2}{2}} \G\left(\frac{k-1-s+a}{2}\right)}.
\end{align*}
We have, on all classical points
\begin{align*}
\frac{L_p(\kappa,\kappa'){B(\kappa,\kappa')}}{L^+_p(\kappa,\kappa'){A(\kappa,\kappa')}}& =
 \frac{\pi^{2k-s-2}s ! G(\xi\eps\omega^s) C(\xi\eps\omega^{-s})^{s}  {D'}^{\frac{s-\beta}{2}} 2^{-2s -k -\frac{1}{2}} 2^{-(2s+5-5k+\frac{1}{2})} p^{ns}\Omega(k-s-1)} 
{\pi^{s+1}(2k-3-s)!G(\psi\xi^{-1}\eps^{-1}\omega^{-s})G(\psi^2\xi^{-1}\eps^{-1}\omega^{-s}) C_0^{2k-s-1}{D'_0}^{k-\beta-\frac{2s-3}{4}} p^{n(2k-s-3)} } \times \\
 & \times \frac{L(s+1,\mathrm{Sym}^2(f),\xi^{-1}\eps^{-1}\omega^{-s})}{L(2k-2-s,\mathrm{Sym}^2(f),\psi^2\xi\eps\omega^{s})} \\
 & = \frac{p^{n(3(s+1)+3k+2)}G(\xi\eps\omega^s) 2^{{4k-4s-6}} {C(\xi\eps\omega^{-s})}^{s}  {D'}^{\frac{s-\beta}{2}}} {G(\psi\xi^{-1}\eps^{-1}\omega^{-s})G(\psi^2\xi^{-1}\eps^{-1}\omega^{-s})C_0^{2k-s-1}{D'_0}^{k-\beta-\frac{2s-3}{4}}} \beps(s-k+2,\hat{\pi}(f),\xi^{-1}\eps^{-1}\omega^{-s}\psi).
\end{align*}
To conclude we use \cite[Lemma 1.4]{Sc} and the relations 
\begin{align*}
p^n = G(\tilde{\psi})G(\tilde{\psi}^{-1}), \:\:\:\: G(\psi_1\psi_2)=\psi_1(C_2)\psi_2(C_1)G(\psi_1)G(\psi_2),
\end{align*}
for $\tilde{\psi}$ a character of conductor $p^n$ and $\psi_i$ a character of conductor $C_i$, with $(C_1,C_2)=1$.
\end{proof}
\begin{prop}\label{disjpole}
The elements $ A(\kappa,\kappa')$ and $ B(\kappa,\kappa')$ are mutually coprime in $\A(\U\times \W) $.
\end{prop}
\begin{proof}
We follow closely the proof of \cite[\S 3.1]{DD}.\\
During the proof of this proposition we shall identify $\A(\U\times \W)$ with $\A(\U)[[T]]$ and we will see $\A(\U)$ as a $\oo[[S]]$-algebra via $S \mapsto (\kappa \mapsto \kappa(u)-1)$. \\
Consider one of the factors of $A(\kappa,\kappa')$ in which neither $\lambda_q(\kappa)$, nor $\alpha^{\circ}_q(\kappa)$, nor $\beta^{\circ}_q(\kappa)$ appear. Then such a factor belongs to $\oo[[S,T]]$ and a prime factor of it is of the form $(1+T)-z(1+S)$, with $z \in \mu_{p^{\infty}}$.\\
A prime divisor of the excluded factors of $A(\kappa,\kappa')$ is $(1+T) - j(\kappa)$, with $j(\kappa)$ in $\A(\U)$.\\
Similarly, a prime factor of $B(\kappa,\kappa')$ is $(1+T)-z'(1+S)$, with $z \in u^{-1}\mu_{p^{\infty}}$ or $(1+T) - j'(\kappa)$. \\
If a prime elements divides both elements, we must have $z(1+S) = j'(\kappa)$ or $z'(1+S) = j(\kappa) $. We deal with the fist case.   Suppose that this prime elements divides $ {\left(1-(\psi\xi^{-1})_0(q)q^{-2}\frac{\kappa(u^{l_q})}{\kappa'(u^{l_q})}\right)}$ and ${\left(1-(\psi^{-2}\xi)_0(q'){q'}^{2}\alpha^{\circ}_{q'}(\kappa)^2\frac{\kappa'(u^{l_{q'}})}{\kappa(u^{2l_{q'}})}\right)}$. Specializing at any classical point $(\kappa,\kappa')$ we obtain $q^{k-s-2}=\zeta {q'}^{s-2k +2}\alpha_{q'}^{\circ}(\kappa)^{2}$, for $\zeta$ a root of unity. Noticing that $|\alpha_{q'}^{\circ}(\kappa)^{2}|_{\C}={q'}^{k-1}$ we obtain $|q^{k-s-2}|_{\C}=|{q'}^{s-k+1}|_{\C}$, contradiction. 
All the other cases are analogous.
\end{proof}
We can than state the main theorem of the appendix. We exclude the case where $\psi\xi\omega^{-1}$ is quadratic imaginary and $F(\kappa)$ has complex multiplication by the corresponding quadratic field because this case has already been treated in \cite{H6}. Recall the ``denominator'' $H_F(\kappa)$ of $l^r_{F}$ defined at the end of section \ref{nomeasures}.
\begin{theo}
We have a two-variable $p$-adic $L$-function $H_F(\kappa)\LL_p(\kappa,\kappa')$ on $\Cc_{F} \times \W$, holomorphic in the first variable and of logarithmic growth $h=[2 \alpha]+2$  in the second variable such that for all classical points $(\kappa, \kappa')$  we have the following interpolation formula
\begin{align*}
\LL_p(\kappa,\kappa') =  C_{\kappa,\kappa'} E_1(\kappa,\kappa')E_2(\kappa,\kappa') \frac{L(s+1,\mathrm{Sym}^2(F(\kappa)),\xi^{-1}\eps^{-1}\omega^{-s})}{\pi^{s+1}S(F(\kappa))W'(F(\kappa))\lla F^{\circ}(\kappa),F^{\circ}(\kappa)\rra}.
\end{align*}
\end{theo}
\begin{proof}
We pose 
\begin{align*}
\LL_p(\kappa,\kappa') :=L_p(\kappa,\kappa') {A(\kappa,\kappa')}^{-1}.
\end{align*}
We begin by showing that $\LL_p(\kappa,\kappa')$ is holomorphic. We know from the definition of $L_p(\kappa,\kappa')$ and Proposition \ref{noCM} that all the poles of $L_p(\kappa,\kappa')$ are controlled by $H_F(\kappa)$, that is $H_F(\kappa)L_p(\kappa,\kappa')$ is holomorphic in $\kappa$.\\
We have moreover that ${A(\kappa,\kappa')}^{-1}$ brings no extra poles; indeed, because of the functional equation of Proposition \ref{padicFE}, a zero of $A(\kappa,\kappa')$ induces a pole of $H_F(\kappa)L^+_p(\kappa,\kappa') {B(\kappa,\kappa')}^{-1}$. But the only poles of the latter could be the zeros of ${B(\kappa,\kappa')}$. Proposition \ref{disjpole} tells us that the zeros of $A(\kappa,\kappa')$ and $B(\kappa,\kappa')$ are disjoint and we are done.\\
To conclude, we have to show the interpolation formula at zeros of $A(\kappa,\kappa')$; for this, it is enough to combine Proposition \ref{padicFE} and Theorem \ref{T3}.
\end{proof}

\Addresses
\end{document}